\documentclass[12pt, letterpaper]{article}

\usepackage[a4paper, total={6.7in, 9.7in}]{geometry}

\usepackage{amsmath,amssymb}
\usepackage{amsthm}
\usepackage{hyperref}
\usepackage{enumitem}
\usepackage{array,multirow}
\usepackage{setspace}
\usepackage{scrextend}
\usepackage{graphicx,wrapfig,amsfonts, eso-pic}
\usepackage{tikz}
\usetikzlibrary{decorations.pathreplacing,decorations.markings}
\usepackage{setspace}
\usepackage{scrextend}
\usetikzlibrary{graphs}
\usetikzlibrary{decorations.pathreplacing,decorations.markings}
\usetikzlibrary{graphs}
\usetikzlibrary{arrows}
\usepackage{graphicx}
\usepackage{caption}
\usepackage{subcaption}
\usepackage{listings}
\newcommand{\N}{\mathbb{N}}
\newcommand{\Z}{\mathbb{Z}}

\newcommand{\eps}{\varepsilon}
\newtheorem{theorem}{Theorem}

\newtheorem{corollary}[theorem]{Corollary}

\newtheorem{lemma}[theorem]{Lemma}

\begin{document}

\title{\textbf{Constructing Triangle Decomposable Multigraphs with Minimum Multi-edges}}
\author{C. M. Mynhardt\footnote{Funded by a Discovery Grant from the Natural Sciences and Engineering Research Council of Canada
} and A. K. Wright\footnote{Based on work done for an Honours Project at the University of Victoria} \\Department of Mathematics and Statistics\\University of Victoria, Victoria, BC, \textsc{Canada}}
\date{}
\maketitle

\begin{abstract}
\noindent
We study triangle decompositions of graphs. We consider constructions of classes of graphs where every edge lies on a triangle and the addition of the minimum number of multiple edges between already adjacent vertices results in a strongly triangle divisible graph that is also triangle decomposable. We explore several classes of planar graphs as well as a class of toroidal graphs. 
\end{abstract}

\section{Introduction}

The study of triangle decompositions falls under the larger umbrella of graph decompositions. Graph decompositions involve partitioning the edge set of a graph into subsets that induce a predefined type of subgraph. The research into other types of decompositions of graphs, including other cliques, bipartite graphs, and trees, is vast. A graph $G$ that is decomposable into subgraphs $H$ can also be thought of as a graph generated by $H$. Formally, a graph or multigraph is said to be \emph{triangle decomposable} if its edge set can be partitioned such that the subgraph induced by each set in the partition is isomorphic to $K_3$. \\

The study of triangle decompositions is rather old, beginning with the following problem from W.S.B Woolhouse in 1844 (\cite{woolhouse} as seen in \cite{kieka}): 
\begin{addmargin}[1em]{2em}
\center{\small{``Determine the number of combinations that can be made of $n$ symbols, $p$ symbols in each; with this limitation, that no combination of $q$ symbols which may appear in any one of them shall be repeated in any other." }}
\end{addmargin}
Structures of the above type where $p = 3$ and $q = 2$ are known as \emph{Steiner triple systems}, named after the Swiss mathematician Jacob Steiner. The existence of these systems implies that complete graphs of order $n$, where $n \equiv 1, 3 \text{ (mod 3)}$, are triangle decomposable.   \\

In 1968 Folkman, as cited in \cite{Spencer}, discussed the desirability of characterizing triangle decomposable graphs. However, in 1970 Nash-Williams \cite{nash} mentioned that such a characterization would be difficult because of its connection to the Four Colour Conjecture, now known as the Four Colour Theorem. The Four Colour Theorem states that all planar graphs can be properly coloured using at most four colours. The Four Colour Theorem was eventually proved in 1976 by Appel and Haken, using computer aided methods \cite{appel}. The simplest proof to date was found in 2005 by Georges Gontheir, also using computer aided methods. \\

The Four Colour Theorem is equivalent to the statement that every bridgeless planar cubic graph is 3-edge colourable (see Theorem 9.12 in \cite{Bondy}). This is because the dual of any maximal planar graph of order 3 or more is a bridgeless planar cubic graph. Suppose $G$ is a bridgeless planar cubic graph that is 3-edge colourable. Then it follows that the join of $G$ and $\overline{K_3}$ is triangle decomposable, that is $H \cong G \vee \overline{K_3}$ is triangle decomposable. \\

\begin{figure}[h]
\centering
\begin{tikzpicture}
\hspace{-6cm}
\begin{scope}
[scale=.75,auto=right,every node/.style={circle,fill=gray!30},nodes={circle,draw, minimum size=.01cm}]
\foreach \lab/\ang in{1/90,2/150,3/210,4/270,5/330,6/30}
\node(\lab)at(\ang:3){};
\node(7)at(0,0){};
\draw[blue, very thick](1) -- (7);
\draw[red, very thick](1) -- (3);
\draw[green, very thick](1) -- (5);
\draw[red, very thick](5) -- (7);
\draw[blue, very thick](5) -- (3);
\draw[green, very thick](7) -- (3);
\draw(1) -- (2) -- (3) -- (4) -- (5) -- (6) -- (1);
\draw(2) -- (7);
\draw(4) -- (7);
\draw(6) -- (7);
\draw(6) to[in=50, out =190] (3);
\draw(2) to[in=130, out =-10] (5);
\draw(1) to[out=290, in =70] (4);
\end{scope}
\hspace{3cm}
$\rightarrow$
\hspace{3cm}
\begin{scope}
[scale=.75,auto=right,every node/.style={circle,fill=gray!30},nodes={circle,draw, minimum size=.01cm}]
\foreach \lab/\ang in{1/90,2/150,3/210,4/270,5/330,6/30}
\node(\lab)at(\ang:3){};
\node(7)at(0,0){};
\draw[blue, very thick](1) -- (7);
\draw[red, very thick](1) -- (3);
\draw[green, very thick](1) -- (5);
\draw[red, very thick](5) -- (7);
\draw[blue, very thick](5) -- (3);
\draw[green, very thick](7) -- (3);
\draw[very thick, magenta](1) -- (2);
\draw[very thick, magenta](2) -- (3);
\draw[very thick, blue](3)-- (4) -- (5);
\draw[very thick, green](5) -- (6) -- (1);
\draw[very thick, red](2) -- (7);
\draw[very thick, cyan](4) -- (7);
\draw[very thick, olive](6) -- (7);
\draw[very thick, olive](6) to[in=50, out =190] (3);
\draw(2)[very thick, red] to[in=130, out =-10] (5);
\draw[very thick, cyan](1) to[out=290, in =70] (4);
\end{scope}
\hspace{3cm}
$\rightarrow$
\hspace{3cm}
\begin{scope}
[scale=.75,auto=right,every node/.style={circle,fill=gray!30},nodes={circle,draw, minimum size=.01cm}]
\foreach \lab/\ang in{1/90,2/150,3/210,4/270,5/330,6/30}
\node(\lab)at(\ang:3){};
\node(7)at(0,0){};
\draw[cyan, very thick](1) -- (7);
\draw[magenta, very thick](1) -- (3);
\draw[green, very thick](1) -- (5);
\draw[red, very thick](5) -- (7);
\draw[blue, very thick](5) -- (3);
\draw[olive, very thick](7) -- (3);
\draw[very thick, magenta](1) -- (2);
\draw[very thick, magenta](2) -- (3);
\draw[very thick, blue](3)-- (4) -- (5);
\draw[very thick, green](5) -- (6) -- (1);
\draw[very thick, red](2) -- (7);
\draw[very thick, cyan](4) -- (7);
\draw[very thick, olive](6) -- (7);
\draw[very thick, olive](6) to[in=50, out =190] (3);
\draw(2)[very thick, red] to[in=130, out =-10] (5);
\draw[very thick, cyan](1) to[out=290, in =70] (4);
\end{scope}
\end{tikzpicture}
\caption{An example of a triangle decomposition of $G \vee \overline{K_3}$ where $G \cong K_4$. }
\end{figure}
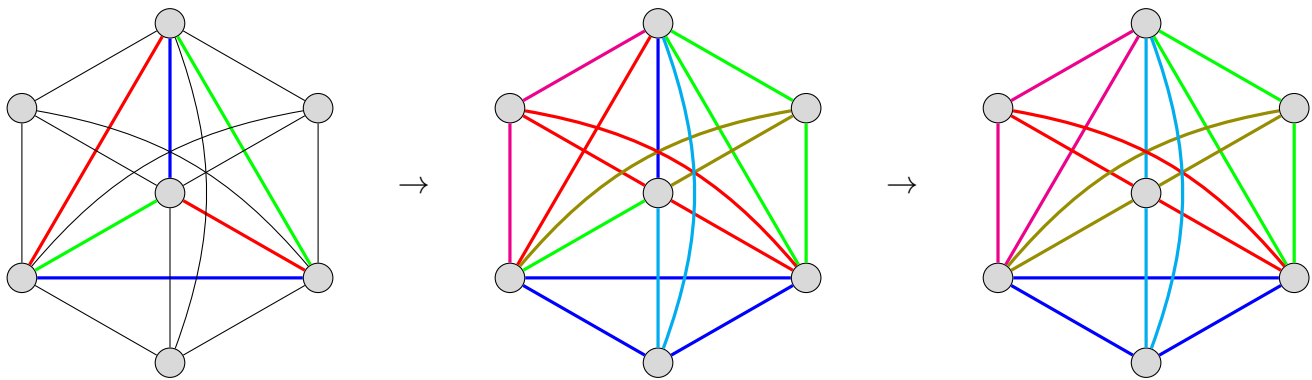

Generally, if $G$ is a bridgeless planar cubic graph then vertices $u, v, w$ can be added to $G$ such that $u, v, w$ are all adjacent to every vertex in $G$ but not each other to form a new graph $H$. A triangle decomposition of $H$ can be formed by choosing, for each edge  $e$ in $G$, one of the vertices $u, v, w$ and colouring the two edges that connect $e$ to the chosen vertex the same colour as that of $e$. The 3-edge colouring of $G$ ensures this method produces a triangle decomposition of $H$. See Figure 1. \\

This connection with the Four Colour Theorem prompted Nash-Williams to look at a  less ambitious problem by instead examining a ``reasonable" graph class $\mathcal{C}$, in which every graph is triangle decomposable. A graph that is Eulerian, has an edge set with cardinality divisible by 3, and all edges of the graph are contained in at least one triangle, is said to be \emph{strongly $K_3$-divisible}. Nash-Williams conjectured such a class $\mathcal{C}$ consists of graphs $G$ that are all sufficiently large, are strongly $K_3$-divisible and have minimum degree $\delta \geq \frac{3}{4}|V(G)|$. The existence of this class is known as the Nash-Williams conjecture. This conjecture is one of the principal parts of the study of triangle decompositions, as it remains open. The best known result comes from Delcourt and Postle \cite{postle} and is as follows: 
\begin{addmargin}[1em]{2em}
\begin{center}\small{``Let $G$ be a $K_3$-divisible graph with $n$ vertices and minimum degree $\delta(G) \geq (\frac{7+\sqrt{21}}{14} + \epsilon)n$. If $n$ is sufficiently large, then $G$ has a triangle decomposition for any $\epsilon > 0$." }
\end{center}
\end{addmargin}
Results aiming to prove the Nash-Williams conjecture or improve the lower bound for $\delta$ dominate much of the study of triangle decompositions. This project will not go into further detail regarding these probabilistic problems but they are worth discussing due to their importance to the topic. \\

Niezen in 2020 \cite{joey} posed a question regarding the assignment of edge multiplicities of graphs in order to create multi-graphs with triangle decompositions. Niezen noted that one simple case is the assignment of edge multiplicities  $\{0,1,...{v \choose 2} -1\}$ to a complete graph $K_v$, as this problem is equivalent to the existence of Sarvate-Beam triple systems. Niezen stated that, for the graph $K_v$, the resulting graph, when adding multiplicities $\{0,1,...{v \choose 2} -1\}$, contains a triangle decomposition as long as $v \equiv 0, 1 \text{ (mod 3)}$.  When looking at a general graph $G$, the question of what edge multiplicities can result in a triangle decomposable graph remains. This project seeks to explore this question by finding constructions of graphs such that adding the minimum number of edges to a graph $G$ to form a strongly $K_3$-divisible multigraph $H$ results in $H$ being triangle decomposable. \\

\section{Notation, tools, and outline}

In 2015, Mynhardt and Van Bommel \cite{kieka} found a necessary and sufficient condition for triangle decomposable planar graphs. They looked at the \emph{depletion} of multigraphs, which are graphs where some number of duplicate and faceless triangles are deleted. A triangle is \emph{faceless} if no plane embedding exists where the triangle is a face of the graph. 

\begin{theorem}[Theorem 1 in \cite{kieka}]
A planar multigraph G is triangle decomposable if and only if some depletion $G_{\Delta}$ of $G$ has a plane embedding whose dual is a bipartite multigraph in which all vertices of some partite set have degree $3$. 
\end{theorem}

The above theorem has the following corollary, which we will frequently use in our constructions as it is quite simple to check whether a graph is Eulerian. 
\begin{corollary}[Corollary 5 in \cite{kieka}]
A maximal planar graph is triangle decomposable if and only if it is Eulerian. 
\end{corollary}

We consider constructions of classes of graphs $G$ where every edge lies on a triangle and the addition of the minimum number of multi-edges between already adjacent vertices results in a strongly triangle divisible graph that is also triangle decomposable. We use $\eps_\Delta(G)$ to denote the minimum number of multi-edges to be added to a graph $G$ between already adjacent vertices to form a triangle decomposable multigraph. \\

We begin, in Section 3.1, with \emph{maximal outerplanar} graphs of order at least 3. An \emph{outerplanar} graph is a graph with a plane embedding such that all vertices of the graph lie on the boundary of one face. We say an outplanar graph is \emph{maximal outerplanar} if the addition of any edge between two non-adjacent vertices results in a graph that is not outerplanar. All maximal outerplanar graphs of order at least 3 are Hamiltonian, by definition, and so we consider this class of graphs from the viewpoint of triangulating the inside of an $n$-cycle. We show that for all $n \geq 3$, we can construct a maximal outerplanar graph with the minimum number of edges that is triangle decomposable. In contrast, we also show that for all $n \geq 3$ there exists a graph that requires the maximum, namely $n -3$, edges of multiplicity 2 in order to be triangle decomposable. In section 3.2 we define and briefly consider $k$-outerplanar graphs. \\

A natural progression from maximal outerplanar graphs is Hamiltonian, maximal planar graphs. We describe the construction of Eulerian, Hamiltonian, maximal planar graphs as, from Corollary 2, this is an equivalent problem to constructing triangle decomposable maximal planar graphs. In Section 3.3 we show that for $n \geq 3$ and $n \neq 4,5,7$, there are no edge multiplicities required to form a triangle decomposable maximal outerplanar, Hamilonian graph on $n$ vertices. \\

In Section 3.4 we look at are simple-clique 2-trees. Generally, a \emph{k-tree} is a graph with $n \geq k+1$ vertices that is created by first forming a $K_k$ and then adding additional vertices by connecting each new vertex to an existing $K_{k-1}$ in the graph. A \emph{simple-clique k-tree} is a $k$-tree where when adding new vertices, each existing $K_{k-1}$ can be used at most once. Thus, a simple-clique 2-tree is identically a 2-tree with no edge on 3 or more triangles. This class of graphs is natural to look at as maximal outerplanar graphs are exactly the class of \emph{simple-clique 2-trees}. It is well known that the forbidden minors of maximal outerplanar graphs are $K_4$ and $K_{2,3}$. These are exactly the forbidden minors of a simple-clique 2-tree \cite{sc3}. Therefore, we can extend our results for maximal outerplanar graphs to this class. We reprove our theorems for maximal outerplanar graphs using a different method that follows from the construction method of all simple-clique 2-trees. \\

In Section 3.5 we look at simple-clique 3-trees. All simple-clique 3-trees are maximal planar \cite{sc3}. The reverse inclusion however does not hold, as simple-clique 3-trees are not Eulerian while maximal planar graphs can be. See Figure 2. \\

\begin{figure}[th!]
\vspace{-1.5cm}
\centering
\begin{tikzpicture}
[scale=1,auto=right,every node/.style={circle,fill=gray!30},nodes={circle,draw, minimum size=.01cm}]
\foreach \lab/\ang in {a/0,b/60,c/120,d/180,e/240,f/300}
{\node(\lab) at(\ang:2){};}
\draw[ ] (a) -- (b) -- (c) -- (a);
\draw[ ] (c) -- (d) -- (e) -- (c);
\draw[ ] (e) -- (f) -- (a) -- (e);
\draw[ ] (b) to[out = 120, in = 140, looseness = 1.2] (d);
\draw[ ] (d) to[out = 220, in = 260, looseness = 1.2] (f);
\draw[ ] (b) to[out = 30, in = -30, looseness = 1.2] (f);
\end{tikzpicture}
\vspace{-1cm}
\caption{The given graph is maximal planar, but is not a simple-clique 3-tree as no vertex of degree 3 exists.}
\end{figure}
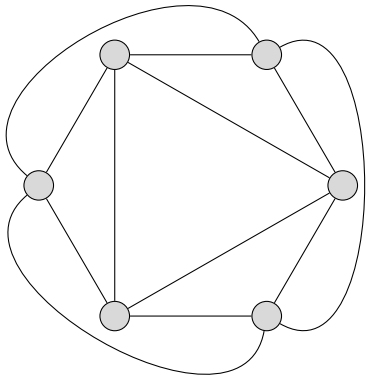

The construction we use to form Eulerian simple-clique 2-trees is relatively simple, but for simple-clique 3-trees the situation complicates. This is because a simple-clique 3-tree always has at least two vertices of odd degree, by construction. Therefore, they must be augmented by at least three multi-edges to form triangle decomposable graphs. We show that  for each $n \geq 4$ there exists a simple-clique 3-tree on $n$ vertices that requires exactly three multi-edges. \\

Finally, in Section 4 we extend the problem to look at toroidal graphs. We consider graphs that are triangulations of the torus and have all edges lying on a single face. We show constructions of graphs for a few given orders $n$ to give an upper bound on the number of multi-edges required to form a triangle decomposable toroidal graphs of this class. \\

We end with concluding remarks and open problems in Section 5. 

\section{Planar Graph Classes}

We explore a variety of graph classes, specifically focusing on planar graphs in this Section. 

\subsection{Maximal outerplanar graphs}

We denote the set of all maximal outerplanar graphs with order $n$ by $\mathcal{OP}_n$. We begin with the following simple lemma -- see \cite{textbook}, for example. 

\begin{lemma}[Theorem 6.26 in \cite{textbook}]
{A maximal outerplanar graph has size $2n - 3$.}\end{lemma}

Recall that $\eps_\Delta(G)$ is the minimum number of multi-edges that need to be added between already adjacent vertices of a graph $G$ to form a triangle decomposable multigraph. That is, $ \eps_\Delta(G) = 0$ if and only if $G$ is triangle decomposable. We define $\eps_{\Delta}(\mathcal{OP}_{n})=\min\{\eps_{\Delta}%
(G):G\in\mathcal{OP}_{n}\}$. \\

The following theorem proves the existence of triangle decomposable outerplanar graphs of order $n$, where $n$ is divisible by 3. Further, it shows that if $n$ is not divisible by 3, then a triangle decomposable, maximal outerplanar graph of order $n$ does not exist. The proof of the theorem relies on the structure of maximal outerplanar graphs. For an alternative proof that utilizes the construction method of simple-clique 2-trees, see Section 3.4. 

\begin{theorem}For $n\geq 3$, $\eps_{\Delta}(\mathcal{OP}_{n}) = 0$ if and only if $n \equiv 0 \text{\emph{ (mod 3)}}$.\end{theorem}

\begin{proof}
Suppose first that $\eps_{\Delta}(\mathcal{OP}_{n}) = 0$. Then, there exists $G \in \mathcal{OP}_n$ such that $\eps_{\Delta}(G) = 0$. Let us fix such a $G$. Then $|E(G)| \equiv 0 \text{ (mod 3)}$. As $G \in \mathcal{OP}_n$, it follows that $E(G) = 2n - 3$. Therefore, $E(G) \equiv 2n - 3 \text{ (mod 3)}$ implies $n \equiv 0 \text{ (mod 3)}$, as desired. \\

Conversely, suppose $n \equiv 0 \text{ (mod 3)}$. We show there exists some graph $G \in \mathcal{OP}_n$ such that $G$ is triangle decomposable. It is convenient to look at four cases, one for each congruence class mod 4. We proceed by induction on $n$. \\

First, consider the following graphs on $n$ vertices, where $n = 3,6, 9, 12$. The order of each of these graphs represents a different congruence class mod 4. 

$$
n = 3: \\
\begin{tikzpicture}
[scale=.5,auto=right,every node/.style={circle,fill=gray!30},nodes={circle,draw, minimum size=.01cm}]
\foreach \lab/\ang in {a/90,b/210,c/330}
{\node(\lab) at(\ang:2){};}
\draw[very thick, red]  (a) -- (b) -- (c) -- (a);
\end{tikzpicture} \hspace{5mm}
n = 6: 
\begin{tikzpicture}
[scale=.5,auto=right,every node/.style={circle,fill=gray!30},nodes={circle,draw, minimum size=.01cm}]
\foreach \lab/\ang in {a/0,b/60,c/120,d/180,e/240,f/300}
{\node(\lab) at(\ang:2){};}
\draw[very thick, red] (a) -- (b) -- (c) -- (a);
\draw[very thick, green] (c) -- (d) -- (e) -- (c);
\draw[very thick, blue] (e) -- (f) -- (a) -- (e);
\end{tikzpicture}$$

$$
n = 9: \\
\begin{tikzpicture}
[scale=.5,auto=right,every node/.style={circle,fill=gray!30},nodes={circle,draw, minimum size=.01cm}]
\foreach \lab/\ang in {a/0,b/40,c/80,d/120,e/160,f/200,g/240,h/280,i/320}
{\node(\lab) at(\ang:3){};}
\draw[very thick, red] (a) -- (b) -- (c) -- (a);
\draw[very thick, green] (c) -- (d) -- (e) -- (c);
\draw[very thick, blue] (e) -- (f) -- (g) -- (e);
\draw[very thick, black] (g) -- (h) -- (i) -- (g);
\draw[very thick, yellow] (a) -- (i) -- (e) -- (a);
\end{tikzpicture} \hspace{5mm}
n = 12: \\
\begin{tikzpicture}
[scale=.5,auto=right,every node/.style={circle,fill=gray!30},nodes={circle,draw, minimum size=.01cm}]
\foreach \lab/\ang in {a/0,b/30,c/60,d/90,e/120,f/150,g/180,h/210,i/240,j/270,k/300,l/330}
{\node(\lab) at(\ang:5){};}
\draw[very thick, red] (a) -- (b) -- (c) -- (a);
\draw[very thick, green] (c) -- (d) -- (e) -- (c);
\draw[very thick, pink] (e) -- (f) --(g) -- (e);
\draw[very thick, blue] (g) -- (h) -- (i) -- (g);
\draw[very thick, orange] (i) -- (j) -- (k) -- (i);
\draw[very thick, cyan] (k) -- (l) -- (a) -- (k);
\draw[very thick, yellow] (a) -- (e) -- (i) -- (a);
\end{tikzpicture}$$ \\

Next, assume there exists a maximal outerplanar graph $G$ of order $3(k-1)$ that is triangle decomposable, for some $k \geq 2$. Recall that all maximal outerplanar graphs of order 3 or more are Hamiltonian, so this assumption is equivalent to saying the inside of any cycle of length $3(k-1)$ can be triangulated to form a triangle decomposable maximal outerplanar graph. \\

For the induction step, we consider four cases, one for each congruence class mod 4. \\

\emph{Case 1:} Suppose $k = 4r$ for some $r \in \N$ where $r < k$. Then, $n = 12r$. \\

Begin by forming a cycle $C_{12r}$. Label the vertices going around the cycle $v_1,v_2,...v_{12r}$. Form triangles around the inner edge of the cycle by connecting every second vertex with an edge, starting at $v_1$. This adds edges between $v_1$ and $v_3$, $v_3$ and $v_5$,... and, $v_{12r-1}$ and $v_1$. These triangles are included in the triangle decomposition of the graph. Notice that the innermost cycle of this graph has length $6r$. 

$$\begin{tikzpicture}
[scale=.4,auto=right,every node/.style={circle,fill=gray!30,draw, inner sep= 2pt}]
\foreach \labb/\lab /\ang in {1/$v_1$/90, 2/$v_2$/60, 3/$v_3$/30,4/$v_4$/0, 5/$v_{12r}$/120,6/$v_{12r-1}$/150,7/$v_{12r-2}$/180}
{\node(\labb)at(\ang:10){\tiny\lab};}
\node(8)at(210:10){\tiny$v_{12r-3}$};
\node(9)at(-30:10){\tiny$v_5$};
\node[fill=none,draw=none](10)at(230:9){};
\node[fill=none,draw=none](11)at(230:10){};
\node[fill=none,draw=none](12)at(-50:9){};
\node[fill=none,draw=none](13)at(-50:10){};
\draw[dashed,red,very thick](10) -- (8);
\draw[dashed](9) -- (13);
\draw[dashed,red,very thick](12) -- (9);
\draw[dashed](8) -- (11);
\draw(1)--(2)--(3)--(4);
\draw(7)--(6)--(5)--(1);
\draw(7) -- (8);
\draw(4)--(9);
\draw[very thick, red, bend right = 60] (1) -- (3);
\draw[very thick, red] (9) -- (3);
\draw[very thick, red] (8) -- (6);
\draw[very thick, red] (1) -- (6);
\end{tikzpicture}$$

Again, connect every other vertex on the inner cycle, starting at $v_1$. This forms a cycle of length $3r$, where no edges have been used in the triangle decomposition yet. Therefore, by the induction hypothesis, additional edges can be added within this cycle to form a triangle decomposable graph. 

$$\begin{tikzpicture}
[scale=.4,auto=right,every node/.style={circle,fill=gray!30,draw, inner sep= 2pt}]
\foreach \labb/\lab /\ang in {1/$v_1$/90, 2/$v_2$/60, 3/$v_3$/30,4/$v_4$/0, 5/$v_{12r}$/120,6/$v_{12r-1}$/150,7/$v_{12r-2}$/180}
{\node(\labb)at(\ang:10){\tiny\lab};}
\node(8)at(210:10){\tiny$v_{12r-3}$};
\node(9)at(-30:10){\tiny$v_5$};
\node[fill=none,draw=none](10)at(230:9){};
\node[fill=none,draw=none](11)at(230:10){};
\node[fill=none,draw=none](12)at(-50:9){};
\node[fill=none,draw=none](13)at(-50:10){};
\draw[dashed,red,very thick](10) -- (8);
\draw[dashed](9) -- (13);
\draw[dashed,red,very thick](12) -- (9);
\draw[dashed](8) -- (11);
\draw(1)--(2)--(3)--(4);
\draw(7)--(6)--(5)--(1);
\draw(7) -- (8);
\draw(4)--(9);
\draw[very thick, red] (1) -- (3);
\draw[very thick, red] (9) -- (3);
\draw[very thick, red] (8) -- (6);
\draw[very thick, red] (1) -- (6);
\node[fill=none,draw=none](14)at(-50:8){};
\node[fill=none,draw=none](15)at(230:8){};
\draw[very thick, blue] (1) -- (9);
\draw[very thick, blue] (1) -- (8);
\draw[very thick, blue,dashed] (9) -- (14);
\draw[very thick, blue,dashed] (8) -- (15);
\end{tikzpicture}$$

\emph{Case 2:} Suppose $k = 4r + 1$ for some $r \in \N$, where $r < k$. Then, $n = 12r + 3$. \\

Begin by forming a cycle $C_{12r+3}$. Label the vertices going around the cycle $v_1,v_2,...v_{12r+3}$. Starting at $v_1$, connect every other vertex with an edge, stopping at $v_{12r+3}$. All edges in the graph, other than $v_{1}v_{12r+3}$, now lie on a triangle. To form a triangle with the edge $v_{1}v_{12r+3}$, connect $v_1$ and $v_{12r+3}$ to the vertex $v_{12r-1}$. All edges in the graph now lie on a triangle. These triangles are included in the triangle decomposition of the graph. Notice the innermost cycle of the graph now has order $6r$. 

$$\begin{tikzpicture}
[scale=.67,auto=right,every node/.style={circle,fill=gray!30, inner sep= 2pt,draw}]
\foreach \labb/\lab /\ang in {1/$v_1$/90,2/$v_2$/70,3/$v_3$/50,4/$v_4$/30,5/$v_5$/10,6/$v_6$/-10,7/$v_{12r+3}$/110,8/$v_{12r+2}$/130,9/$v_{12r+1}$/150,10/$v_{12r}$/170,11/$v_{12r-1}$/190}
\node(\labb)at(\ang:7){\tiny\lab};
\node[fill=none,draw = none](12)at(-20:7){};
\node[fill=none,draw = none](13)at(-20:6.5){};
\node[fill=none,draw = none](14)at(200:7){};
\node[fill=none,draw = none](15)at(200:6.5){};
\draw[dashed](12)--(6);
\draw[dashed,very thick, red](5)to[bend right=15](13);
\draw[dashed](11)--(14);
\draw[dashed,very thick, red](11)--(15);
\draw (11) -- (10) -- (9)--(8)--(7)--(1)--(2)--(3)--(4)--(5)--(6);
\draw[very thick, red] (1) to[bend right=15](3) ;
\draw[very thick, red] (3) to[bend right=15] (5);
\draw[very thick, red] (11) to[bend right=15] (9);
\draw[very thick, red] (9)to[bend right=15] (7);
\draw[very thick, red] (1) to[bend left=15] (11) to[bend right=15] (7);
\end{tikzpicture}$$

Again, connect every other vertex on the inner cycle, starting at $v_1$. This forms a cycle of length $3r$, where no edges have been used in the triangle decomposition yet. Therefore, by the induction hypothesis, additional edges can be added within the cycle to form a triangle decomposable graph. \\

$$\begin{tikzpicture}
[scale=.67,auto=right,every node/.style={circle,fill=gray!30, inner sep= 2pt,draw}]
\foreach \labb/\lab /\ang in {1/$v_1$/90,2/$v_2$/70,3/$v_3$/50,4/$v_4$/30,5/$v_5$/10,6/$v_6$/-10,7/$v_{12r+3}$/110,8/$v_{12r+2}$/130,9/$v_{12r+1}$/150,10/$v_{12r}$/170,11/$v_{12r-1}$/190}
\node(\labb)at(\ang:7){\tiny\lab};
\node[fill=none,draw = none](12)at(-20:7){};
\node[fill=none,draw = none](13)at(-20:6.5){};
\node[fill=none,draw = none](14)at(200:7){};
\node[fill=none,draw = none](15)at(200:6.5){};
\draw[dashed](12)--(6);
\draw[dashed,very thick, red](5)to[bend right=15](13);
\draw[dashed](11)--(14);
\draw[dashed,very thick, red](11)--(15);
\draw (11) -- (10) -- (9)--(8)--(7)--(1)--(2)--(3)--(4)--(5)--(6);
\draw[very thick, red] (1) to[bend right=15](3) ;
\draw[very thick, red] (3) to[bend right=15] (5);
\draw[very thick, red] (11) to[bend right=15] (9);
\draw[very thick, red] (9)to[bend right=15] (7);
\draw[very thick, red] (1) to[bend left=15] (11) to[bend right=15] (7);
\draw[very thick, blue](1)to[bend right =15](5);
\node[fill=none,draw=none](16)at(200:5.5){};
\draw[very thick, blue,dashed](1)to[bend left =15](16);
\node[fill=none,draw=none](17)at(-20:5.5){};
\draw[dashed,very thick, blue](5)to[bend right =15](17);
\end{tikzpicture}$$ \\~\\

\emph{Case 3:} Suppose $k = 4r + 2$, where $r < k$. Then, $n = 12r + 6$. \\

Begin by forming a cycle of order $C_{12r+6}$. Label the vertices going around the cycle $v_1,v_2,...v_{12r+6}$. Starting at $v_1$, connect every other vertex. Every edge in the graph now lies on a triangle. We include all of these triangles in our triangle decomposition. Notice that the innermost cycle of the graph has length $6r+3$. 

$$\begin{tikzpicture}
[scale=.6,auto=right,every node/.style={circle,fill=gray!30, inner sep = 1pt,draw}]
\foreach \labb/\lab /\ang in {1/$v_1$/90,2/$v_2$/70,3/$v_3$/50,4/$v_4$/30,5/$v_5$/10,6/$v_6$/-10,7/$v_{12r+6}$/110,8/$v_{12r+5}$/130,9/$v_{12r+4}$/150,10/$v_{12r+3}$/170,11/$v_{12r+2}$/190,12/$v_{12r+1}$/210,15/$v_{7}$/-30,13/$v_{12r}$/230,14/$v_{12r-1}$/250}
\node(\labb)at(\ang:7){\tiny\lab};
\draw(14)--(13)--(12)--(11)--(10)--(9)--(8)--(7)--(1)--(2)--(3)--(4)--(5)--(6)--(15);
\draw[red,very thick](1)to[bend right =15] (3);
\draw[red,very thick](3)to[bend right =15](5);
\draw[red,very thick](5)to[bend right =15](15);
\draw[red,very thick](14)to[bend right =15](12);
\draw[red,very thick](12)to[bend right =15](10);
\draw[red,very thick](10)to[bend right =15](8);
\draw[red,very thick](8)to[bend right =15](1);
\node[fill=none, draw = none](16)at(265:7){};
\draw[dashed](14)--(16);
\node[fill=none, draw = none](17)at(265:6){};
\draw[dashed,red,very thick](14)--(17);
\node[fill=none, draw = none](18)at(-40:7){};
\draw[dashed](15)--(18);
\node[fill=none, draw = none](19)at(-40:6){};
\draw[dashed,red,very thick](15)--(19);
\end{tikzpicture}$$
Next, form a triangle with vertices $v_1$, $v_{12r+3}$, and $v_{12r}$. This new triangle is also part of the triangle decomposition. Notice the innermost cycle is now of order $6r$.  
$$\begin{tikzpicture}
[scale=.6,auto=right,every node/.style={circle,fill=gray!30, inner sep = 1pt,draw}]
\foreach \labb/\lab /\ang in {1/$v_1$/90,2/$v_2$/70,3/$v_3$/50,4/$v_4$/30,5/$v_5$/10,6/$v_6$/-10,7/$v_{12r+6}$/110,8/$v_{12r+5}$/130,9/$v_{12r+4}$/150,10/$v_{12r+3}$/170,11/$v_{12r+2}$/190,12/$v_{12r+1}$/210,15/$v_{7}$/-30,13/$v_{12r}$/230,14/$v_{12r-1}$/250}
\node(\labb)at(\ang:7){\tiny\lab};
\draw(14)--(13)--(12)--(11)--(10)--(9)--(8)--(7)--(1)--(2)--(3)--(4)--(5)--(6)--(15);
\draw[red,very thick](1)to[bend right =15] (3);
\draw[red,very thick](3)to[bend right =15](5);
\draw[red,very thick](5)to[bend right =15](15);
\draw[red,very thick](14)to[bend right =15](12);
\draw[red,very thick](12)to[bend right =15](10);
\draw[red,very thick](10)to[bend right =15](8);
\draw[red,very thick](8)to[bend right =15](1);
\node[fill=none, draw = none](16)at(265:7){};
\draw[dashed](14)--(16);
\node[fill=none, draw = none](17)at(265:6){};
\draw[dashed,red,very thick](14)--(17);
\node[fill=none, draw = none](18)at(-40:7){};
\draw[dashed](15)--(18);
\node[fill=none, draw = none](19)at(-40:6){};
\draw[dashed,red,very thick](15)--(19);
\node[fill=none, draw = none](18)at(-40:7){};
\draw[dashed](15)--(18);
\node[fill=none, draw = none](19)at(-40:6){};
\draw[very thick, blue](14)to[bend left =0](1);
\draw[very thick, blue](1)to[bend left =5](10);
\draw[very thick, blue](10)to[bend left =5](14);
\end{tikzpicture}$$

Again, connect every other vertex of the inner cycle starting at $v_1$. This forms a cycle of length $3r$ where no edges have been used in the triangle decomposition yet. As $r < k$, the result follows from the induction hypothesis. 

$$\begin{tikzpicture}
[scale=.6,auto=right,every node/.style={circle,fill=gray!30, inner sep = 1pt,draw}]
\foreach \labb/\lab /\ang in {1/$v_1$/90,2/$v_2$/70,3/$v_3$/50,4/$v_4$/30,5/$v_5$/10,6/$v_6$/-10,7/$v_{12r+6}$/110,8/$v_{12r+5}$/130,9/$v_{12r+4}$/150,10/$v_{12r+3}$/170,11/$v_{12r+2}$/190,12/$v_{12r+1}$/210,15/$v_{7}$/-30,13/$v_{12r}$/230,14/$v_{12r-1}$/250}
\node(\labb)at(\ang:7){\tiny\lab};
\draw(14)--(13)--(12)--(11)--(10)--(9)--(8)--(7)--(1)--(2)--(3)--(4)--(5)--(6)--(15);
\draw[red,very thick](1)to[bend right =15] (3);
\draw[red,very thick](3)to[bend right =15](5);
\draw[red,very thick](5)to[bend right =15](15);
\draw[red,very thick](14)to[bend right =15](12);
\draw[red,very thick](12)to[bend right =15](10);
\draw[red,very thick](10)to[bend right =15](8);
\draw[red,very thick](8)to[bend right =15](1);
\node[fill=none, draw = none](16)at(265:7){};
\draw[dashed](14)--(16);
\node[fill=none, draw = none](17)at(265:6){};
\draw[dashed,red,very thick](14)--(17);
\node[fill=none, draw = none](18)at(-40:7){};
\draw[dashed](15)--(18);
\node[fill=none, draw = none](19)at(-40:6){};
\draw[dashed,red,very thick](15)--(19);
\node[fill=none, draw = none](18)at(-40:7){};
\draw[dashed](15)--(18);
\node[fill=none, draw = none](19)at(-40:6){};
\draw[very thick, blue](14)to[bend left =0](1);
\draw[very thick, blue](1)to[bend left =5](10);
\draw[very thick, blue](10)to[bend left =5](14);
\draw[very thick, cyan] (1) to[bend right =10](5);
\draw[very thick, cyan, dashed](5)to[bend right =10](19);
\draw[very thick, cyan, dashed](1)to[bend left =10](17);
\end{tikzpicture}$$

\emph{Case 4:} Suppose $k = 4r+3$ for some $r \in \N$, where $r < k$. Then, $n = 12r+9$. \\

Begin by forming a cycle $C_{12r+9}$. Similar to Case 2, connect every other vertex in the graph beginning at $v_1$ until $v_{12r+9}$. Every edge other than $v_1v_{12r+9}$ lies in a triangle. Form a triangle between $v_{1}$, $v_{12r+9}$ and $v_{12r+5}$ by adding an edge between $v_1$ and $v_{12r+5}$ and an edge between $v_{12r+9}$ and $v_{12r+5}$. The inner-cycle of the graph has order $6r+3$. 

$$\begin{tikzpicture}
[scale=.7,auto=right,every node/.style={circle,fill=gray!30, inner sep = 1pt,draw}]
\foreach \labb/\lab /\ang in {1/$v_1$/90,2/$v_2$/75,3/$v_3$/60,4/$v_4$/45,5/$v_5$/30,6/$v_6$/15,
7/$v_{12r+9}$/105}
\node[](\labb)at(\ang:7){\tiny\lab};
\foreach \labb/\lab /\ang in {11/$v_{12r+5}$/165,12/$v_{12r+4}$/180,13/$v_{12r+3}$/195,14/$v_{12r+2}$/210,15/$v_{12r+1}$/225,10/$v_{12r+6}$/150,9/$v_{12r+7}$/135,8/$v_{12r+8}$/120}
\node[](\labb)at(\ang:7){\tiny\lab};
\foreach \labb/\lab /\ang in {16/$v_{12r}$/240,17/$v_{12r-1}$/255}
\node[](\labb)at(\ang:7){\tiny\lab};
\node[](19)at(0:7){\tiny$v_7$};
\draw(17)--(16)--(15)--(14)--(13)--(12)--(11)--(10)--(9)--(8)--(7)--(1)--(2)--(3)--(4)--(5)--(6)--(19);
\draw[red,very thick](1)to[bend right = 22](3);
\draw[red,very thick](3)to[bend right = 22](5);
\draw[red,very thick](5)to[bend right = 22](19);
\draw[red,very thick](17)to[bend right = 22](15);
\draw[red,very thick](15)to[bend right = 22](13);
\draw[red,very thick](13)to[bend right = 22](11);
\draw[red,very thick](11)to[bend right = 22](9);
\draw[red,very thick](9)to[bend right = 22](7);
\draw[red,very thick](7)to[bend left = 22](11);
\draw[red,very thick](11)to[bend right = 22](1);
\node[fill=none,draw = none](a)at(-10:7){};
\node[fill=none,draw = none](b)at(-10:6.5){};
\draw[dashed](19) -- (a);
\draw[dashed,red,very thick](19) -- (b);
\node[fill=none,draw = none](a')at(265:7){};
\node[fill=none,draw = none](b')at(265:6.5){};
\draw[dashed](17) -- (a');
\draw[dashed,red,very thick](17) -- (b');
\end{tikzpicture}$$

Form a triangle with vertices $v_1$, $v_{12r+3}$ and $v_{12r-1}$. This new triangle is included in our triangle decomposition. The innermost cycle is of order $6r$. \\

$$\begin{tikzpicture}
[scale=.65,auto=right,every node/.style={circle,fill=gray!30, inner sep = 1pt,draw}]
\foreach \labb/\lab /\ang in {1/$v_1$/90,2/$v_2$/75,3/$v_3$/60,4/$v_4$/45,5/$v_5$/30,6/$v_6$/15,
7/$v_{12r+9}$/105}
\node[](\labb)at(\ang:7){\tiny\lab};
\foreach \labb/\lab /\ang in {11/$v_{12r+5}$/165,12/$v_{12r+4}$/180,13/$v_{12r+3}$/195,14/$v_{12r+2}$/210,15/$v_{12r+1}$/225,10/$v_{12r+6}$/150,9/$v_{12r+7}$/135,8/$v_{12r+8}$/120}
\node[](\labb)at(\ang:7){\tiny\lab};
\foreach \labb/\lab /\ang in {16/$v_{12r}$/240,17/$v_{12r-1}$/255}
\node[](\labb)at(\ang:7){\tiny\lab};
\node[](19)at(0:7){\tiny$v_7$};
\draw(17)--(16)--(15)--(14)--(13)--(12)--(11)--(10)--(9)--(8)--(7)--(1)--(2)--(3)--(4)--(5)--(6)--(19);
\draw[red,very thick](1)to[bend right = 22](3);
\draw[red,very thick](3)to[bend right = 22](5);
\draw[red,very thick](5)to[bend right = 22](19);
\draw[red,very thick](17)to[bend right = 22](15);
\draw[red,very thick](15)to[bend right = 22](13);
\draw[red,very thick](13)to[bend right = 22](11);
\draw[red,very thick](11)to[bend right = 22](9);
\draw[red,very thick](9)to[bend right = 22](7);
\draw[red,very thick](7)to[bend left = 22](11);
\draw[red,very thick](11)to[bend right = 22](1);
\node[fill=none,draw = none](a)at(-10:7){};
\node[fill=none,draw = none](b)at(-10:6.5){};
\draw[dashed](19) -- (a);
\draw[dashed,red,very thick](19) -- (b);
\node[fill=none,draw = none](a')at(265:7){};
\node[fill=none,draw = none](b')at(265:6.5){};
\draw[dashed](17) -- (a');
\draw[dashed,red,very thick](17) -- (b');
\draw[very thick, blue] (1) to[bend left = 22] (13);
\draw[very thick, blue] (13) to[bend left =22] (17);
\draw[very thick, blue] (17) -- (1);
\end{tikzpicture}$$
Connect every other vertex in the innermost cycle, beginning at $v_1$. This will form a cycle of length $3r$ where no edges have been used in the triangle decomposition of the graph. As $r < k$, the result follows from the induction hypothesis. 
$$\begin{tikzpicture}
[scale=.65,auto=right,every node/.style={circle,fill=gray!30, inner sep = 1pt,draw}]
\foreach \labb/\lab /\ang in {1/$v_1$/90,2/$v_2$/75,3/$v_3$/60,4/$v_4$/45,5/$v_5$/30,6/$v_6$/15,
7/$v_{12r+9}$/105}
\node[](\labb)at(\ang:7){\tiny\lab};
\foreach \labb/\lab /\ang in {11/$v_{12r+5}$/165,12/$v_{12r+4}$/180,13/$v_{12r+3}$/195,14/$v_{12r+2}$/210,15/$v_{12r+1}$/225,10/$v_{12r+6}$/150,9/$v_{12r+7}$/135,8/$v_{12r+8}$/120}
\node[](\labb)at(\ang:7){\tiny\lab};
\foreach \labb/\lab /\ang in {16/$v_{12r}$/240,17/$v_{12r-1}$/255}
\node[](\labb)at(\ang:7){\tiny\lab};
\node[](19)at(0:7){\tiny$v_7$};
\draw(17)--(16)--(15)--(14)--(13)--(12)--(11)--(10)--(9)--(8)--(7)--(1)--(2)--(3)--(4)--(5)--(6)--(19);
\draw[red,very thick](1)to[bend right = 22](3);
\draw[red,very thick](3)to[bend right = 22](5);
\draw[red,very thick](5)to[bend right = 22](19);
\draw[red,very thick](17)to[bend right = 22](15);
\draw[red,very thick](15)to[bend right = 22](13);
\draw[red,very thick](13)to[bend right = 22](11);
\draw[red,very thick](11)to[bend right = 22](9);
\draw[red,very thick](9)to[bend right = 22](7);
\draw[red,very thick](7)to[bend left = 22](11);
\draw[red,very thick](11)to[bend right = 22](1);
\node[fill=none,draw = none](a)at(-10:7){};
\node[fill=none,draw = none](b)at(-10:6.5){};
\draw[dashed](19) -- (a);
\draw[dashed,red,very thick](19) -- (b);
\node[fill=none,draw = none](a')at(265:7){};
\node[fill=none,draw = none](b')at(265:6.5){};
\draw[dashed](17) -- (a');
\draw[dashed,red,very thick](17) -- (b');
\draw[very thick, blue] (1) to[bend left = 22] (13);
\draw[very thick, blue] (13) to[bend left =22] (17);
\draw[very thick, blue] (17) -- (1);
\node[fill=none,draw=none](c')at(275:6){};
\draw[dashed](17) -- (a');
\node[fill=none,draw=none](c)at(-10:6){};
\draw[dashed,red,very thick](17) -- (b');
\draw[dashed,cyan,very thick](1) -- (c');
\draw[dashed,cyan,very thick](5) to[bend right = 15](c);
\draw[very thick, cyan] (1)to[bend right = 17](5);
\end{tikzpicture}$$

As Cases 1-4 exhaust all possibilities, the theorem follows by the principle of induction. 

\end{proof}
The following two results imply $\eps_\Delta(\mathcal{OP}_n)$ is dictated by the congruence class of $n$ mod 3. The proofs follow the same method as the proof of Theorem 4. See the appendix. 

\begin{theorem} For all $n \geq 3$, $n \equiv 1 \text{\emph{ (mod 3)}}$ if and only if $\eps_{\Delta}(\mathcal{OP}_{n}) = 1$.\end{theorem}
\begin{theorem} For all $n \geq 3$, $n \equiv 2 \text{\emph{ (mod 3)}}$ if and only if $\eps_{\Delta}(\mathcal{OP}_{n}) = 2$. \end{theorem}

Next, let us consider the maximum number of edges required for any graph in $\mathcal{OP}_n$ to be made triangle decomposable, where we only consider edge multiplicities of order two. We define $\Xi_{\Delta}(\mathcal{OP}_{n})=\max\{\eps_{\Delta}(G):G\in
\mathcal{OP}_{n}\}.$ \\

\begin{theorem}For all $n \geq 3$, {$\Xi_n(\mathcal{OP}_n) = n - 3$.} \end{theorem}

\begin{proof}


First, we construct a graph $G$ on $n \geq 3$ vertices such that $\eps_\Delta(G) = n -3$. Begin by forming a $C_n$. Label the vertices $v_1,v_2,..v_n$. Connect $v_1$ to all other vertices in the graph. This is a maximal outerplanar graph. Notice that all edges on the interior of the cycle lie on two triangles, while all edges on the cycle lie on exactly one triangle. This implies that a multi-edge is required on every interior edge of the graph to form a triangle decomposition. Thus, $\eps_\Delta(G) = n -3$. Therefore, $\Xi_{\Delta}(\mathcal{OP}_{n}) \geq n - 3$. \\

Conversely, notice that  edges on the outer face of maximal outerplanar graphs always lie on exactly one triangle. This implies that adding multi-edges to the outerface is never necessary when adding the minimum required number of multi-edges to a graph to form a triangle decomposition. Thus, $\Xi_{\Delta}(\mathcal{OP}_{n}) \leq n - 3$ as there are $n$ edges on the outside face. The result follows. 
\end{proof}

We have now shown that there exist graphs in $\mathcal{OP}_n$ that require the minimum number of multi-edges and maximum number of multi-edges to form a triangle decomposable graph. We can also show that there exists a graph in $\mathcal{OP}_n$ that requires every possible number of edges between $\eps_\Delta(\mathcal{OP}_n)$ and $\Xi_\Delta(\mathcal{OP}_n)$ that would ensure a graph has size divisible by 3. 

\begin{theorem}{Suppose $\Xi_{\Delta}(\mathcal{OP}_{n})=\varepsilon_{\Delta}(\mathcal{OP}%
_{n})+3k$ for some $k \in \Z$. Then, there exists a graph $G_r$ such that $\eps_\Delta(G_r) = \eps_\Delta(\mathcal{OP}) + 3r$ where $r = 1, 2, ... k - 1$.} \end{theorem}
The proof of Theorem 8 can be found in the appendix.  \\

\subsection{$k$-Outerplanar Graphs}

A simple extension of the class of outerplanar graphs are $k$-outerplanar graphs and so this class of graphs is a natural one to explore next. \\

A 1-\emph{outerplanar graph} is simply an outerplanar graph. For all $k \geq 2$, a graph is \emph{$k$-outerplanar} if it has a planar embedding and when all vertices on the outer-face are removed, it leaves a $(k-1)$-outerplanar graph. We call the removal of vertices on the outerface the \emph{onion peel} of the graph. A $k$-outerplanar graph has $k$ onion peel subgraphs. \\

We define $\mathcal{KOP}_m$ to be the set of all $k$-outerplanar graphs with $m\cdot k$ vertices, that is, each layer has $m$ vertices. The results in Section 3.1 allows us to construct graphs in $\mathcal{KOP}_m$ that require the minimum number of multi-edges to form a triangle decomposition. \\

\begin{theorem}
There exists a graph $G$ in $\mathcal{KOP}_n$ such that $\eps_\Delta(G)$ is the equal to the least residue of $m \text{\emph{ (mod 3)}}$.  
\end{theorem}
\begin{proof}

For $k = 1$, the results follows directly from Theorems 5-7. Assume $k > 1$. \\

From Theorems 5-7 it follows that there exists am outerplanar graph $G'$ such that $G'$ has $m$ vertices where $\eps_\Delta(G')$ is equal to the least residue of $m \text{ (mod 3)}$. Add a $C_{m}$ graph such that $G'$ lies completely inside the cycle. Join the vertices of $C_m$ to those of $G'$ by a perfect matching in such a way that no edges overlap. Next, connect each vertex in $C_{m}$ to the vertex that is clockwise from the vertex it is already connected to in $G'$. This will form triangles around the edge of the graph. This new graph $G''$ is a 2-outerplanar, triangle decomposable graph. We can continue adding vertices in this manner to eventually form a $k$-outerplanar graph that is triangle decomposable. See Figure 3. 
\end{proof}

\begin{figure}[th!]
\centering
\begin{tikzpicture}
\hspace{-3cm}
\vspace{-10cm}
\begin{scope}
[scale=.8,auto=right,every node/.style={circle,fill=gray!30},nodes={circle,draw, minimum size=.001mm,inner sep=2pt}]
\foreach \lab/\ang in {a/0,b/60,c/120,d/180,e/240,f/300}
{\node(\lab) at(\ang:1.5){};}
\draw[very thick, red] (a) -- (b) -- (c) -- (a);
\draw[very thick, green] (c) -- (d) -- (e) -- (c);
\draw[very thick, blue] (e) -- (f) -- (a) -- (e);
\end{scope}
\hspace{3cm}
$\rightarrow$
\hspace{4cm}
\begin{scope}
[scale=.8,auto=right,every node/.style={circle,fill=gray!30},nodes={circle,draw, minimum size=.001mm,inner sep=2pt}]
\foreach \lab/\ang in {a/0,b/60,c/120,d/180,e/240,f/300}
{\node(\lab) at(\ang:1.5){};}
\foreach \lab/\ang in {a'/0,b'/60,c'/120,d'/180,e'/240,f'/300}
{\node(\lab) at(\ang:3){};}
\draw[very thick, red] (a) -- (b) -- (c) -- (a);
\draw[very thick, green] (c) -- (d) -- (e) -- (c);
\draw[very thick, blue] (e) -- (f) -- (a) -- (e);
\draw[very thick, orange] (a') -- (a) -- (b') -- (a');
\draw[very thick, pink] (b') -- (b) -- (c') -- (b');
\draw[very thick, black] (c') -- (c) -- (d') -- (c');
\draw[very thick, cyan] (d') -- (d) -- (e') -- (d');
\draw[very thick, brown] (e') -- (e) -- (f') -- (e');
\draw[very thick, black] (f') -- (f) -- (a') -- (f');
\draw[very thick, yellow] (f') -- (f) -- (a') -- (f');
\end{scope}
\end{tikzpicture}
\caption{An construction of a 2-outerplanar graph with order 12, where each layer has 6 vertices. }
\end{figure}
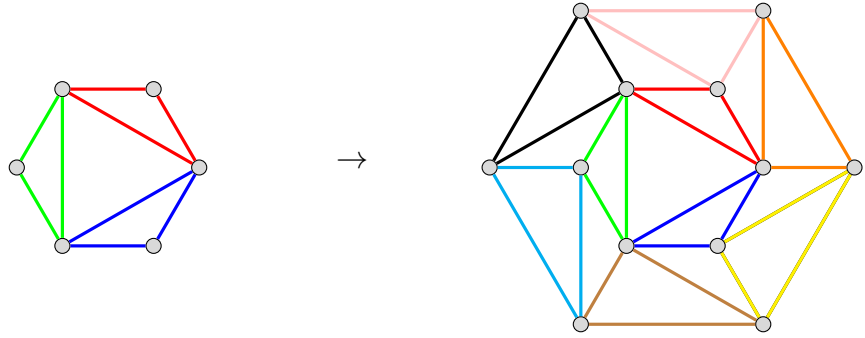
\subsection{Maximal Planar Hamiltonian Graphs}

Having considered different types of maximal outerplanar graphs, it seems natural to wonder whether we could extend our results to maximal planar Hamiltonian graphs as all maximal outerplanar graphs of order at least 3 are Hamiltonian, by definition. We begin with a simple lemma that handles the exceptions of our main theorem for this graph class. \\

\begin{lemma}{There does not exist a maximal planar Hamiltonian graph with order $n$ that is triangle decomposable if $n = 4, 5, 7$.}\end{lemma}
\begin{proof}{ By Euler's identity, a maximal planar graph of order $n \geq 3$ has size $3n -6$. \\

Suppose $n = 4$. Notice that the only maximal planar graph with order 4 is $K_4$. Recall from Mynhardt and Van Bommel \cite{kieka} that maximal planar graphs are triangle decomposable if and only if they are Eulerian. The graph $K_4$ is not Eulerian and so the result follows. \\

Suppose $n = 5$. Then, notice that that the only maximal planar graph with order 5 is $K_5 - e$ (where $e$ is any edge) which is not an Eulerian graph and so the result follows.\\

Suppose $n = 7$. Then, notice that the only possible Eulerian maximal planar graph $G$ of order 7, which has degree sum 30, would contain six vertices of degree 4 and one vertex of degree 6. This implies that there is a cubic graph of order 6 that is an induced subgraph of $G$. There exist only two cubic graphs of order 6, namely $K_{3,3}$ and the following graph, which is $K_2 \square K_3$.
$$\begin{tikzpicture}
[scale=.5,auto=right,every node/.style={circle,fill=gray!30},nodes={circle,draw, minimum size=.01cm, inner sep = 2pt}]
\foreach \lab/\ang in{1/90,2/210,3/330}
\node(\lab)at(\ang:1.5){};
\foreach \lab/\ang in{a/90,b/210,c/330}
\node(\lab)at(\ang:4){};
\draw (1) -- (2) -- (3) -- (1);
\draw (a) -- (b) -- (c) -- (a);
\draw(a) -- (1);
\draw (b) -- (2);
\draw (c) -- (3);
\end{tikzpicture}$$
Notice that no redrawing of this graph allows for the addition of a universal vertex while maintaining planarity. Thus, there is no Eulerian maximal planar graph with order 7. }
\end{proof}

We now prove the existence of a triangle decomposable Hamiltonian maximal planar graph of order $n \geq 6$, $n \neq 7$. \\

\begin{theorem}{For all $n \geq 6$, where $n \neq 7$, there exists a triangle decomposable Hamiltonian maximal planar graph with order $n$. }\end{theorem}
\begin{proof}{

\textbf{Case 1:} $n \geq 6$ and $n \equiv 0 \text{ (mod } 2):$\\

A graph on $n$ vertices can be constructed as follows, see Figure 4:
\begin{enumerate}
\item Form a $C_{n-2}$ graph. Label the vertices around the cycle $v_1$, $v_2$, ... $v_{n-2}$ going around the cycle.  
\item Place an additional vertex $v_{n-1}$ on the inside of the cycle and draw edges between it and all vertices on the cycle.
\item Place an additional vertex $v_n$ outside the cycle and draw edges between it and all vertices on the cycle.
\end{enumerate}

\begin{figure}[h]
\hspace{-2cm}
\begin{tikzpicture}
\begin{scope}
[scale=.35,auto=right,every node/.style={circle,fill=gray!40},nodes={circle,draw, minimum size=.01cm, inner sep = 2pt},every path/.style={}]
\foreach \lab/\ang in {1/90,2/140,3/190,4/240,5/290,{n-2}/30}
\node(\lab)at(\ang:4){\small{$v_{\lab}$}};
\draw({n-2}) -- (1);
\draw[](1) -- (2)--(3)--(4);
\draw(4)--(5);
\draw[ thick,dashed](5) to[out=30,in=290] ({n-2});
\end{scope}
\hspace{3cm}
$\rightarrow$
\hspace{3cm}
\begin{scope}
[scale=.35,auto=right,every node/.style={circle,fill=gray!40},nodes={circle,draw, minimum size=.01cm, inner sep = 2pt},every path/.style={}]
\foreach \lab/\ang in {1/90,2/140,3/190,4/240,5/290,{n-2}/30}
\node(\lab)at(\ang:4){\small{$v_{\lab}$}};
\node(7)at(0,0){\small{$v_{n-1}$}};
\draw({n-2}) -- (1);
\draw(1) -- (2)--(3)--(4);
\draw(4)--(5);
\draw[ thick,dashed](5) to[out=30,in=290] ({n-2});
\draw(1) -- (7);
\draw(7) -- (2) ;
\draw(3) -- (7) -- (4);
\draw(5) -- (7);
\draw(7) -- ({n-2});
\end{scope}
\hspace{3cm}
$\rightarrow$
\hspace{4cm}
\begin{scope}
[scale=.35,auto=right,every node/.style={circle,fill=gray!40},nodes={circle,draw, minimum size=.01cm, inner sep = 2pt},every path/.style={}]
\foreach \lab/\ang in {1/90,2/140,3/190,4/240,5/290,{n-2}/30}
\node(\lab)at(\ang:4){\small{$v_{\lab}$}};
\node(7)at(0,0){\small{$v_{n-1}$}};
\node(8)at(-8,0){\small{$v_{n}$}};
\draw({n-2}) -- (1);
\draw(1) -- (2)--(3)--(4);
\draw(4)--(5);
\draw[ thick,dashed](5) to[out=30,in=290] ({n-2});
\draw(1) -- (7);
\draw(7) -- (2) ;
\draw(3) -- (7) -- (4);
\draw(5) -- (7);
\draw(7) -- ({n-2});
\draw(8) -- (3);
\draw(8) -- (2);
\draw(8) -- (4);
\draw(1) to[out = 170, in = 70] (8);
\draw({n-2}) to[out = 110, in = 85] (8);
\draw(5) to[out = 210, in = 290] (8);
\end{scope}
\end{tikzpicture}\\
\hspace{5cm}
\begin{tikzpicture}
\hspace{1cm}
$\rightarrow$
\hspace{5cm}
\begin{scope}
[scale=.4,auto=right,every node/.style={circle,fill=gray!40},nodes={circle,draw, minimum size=.01cm,, inner sep = 2pt},every path/.style={}]
\foreach \lab/\ang in {1/90,2/140,3/190,4/240,5/290,{n-2}/30}
\node(\lab)at(\ang:4){\scriptsize{$v_{\lab}$}};
\node(7)at(0,0){\scriptsize{$v_{n-1}$}};
\node(8)at(-8,0){\scriptsize{$v_{n}$}};
\draw({n-2}) -- (1);
\draw[red,very thick](1) -- (2)--(3)--(4);
\draw(4)--(5);
\draw[red,very thick,dashed](5) to[out=30,in=290] ({n-2});
\draw[red,very thick](1) -- (7);
\draw(7) -- (2) ;
\draw(3) -- (7) -- (4);
\draw(5) -- (7);
\draw[red,very thick](7) -- ({n-2});
\draw(8) -- (3);
\draw(8) -- (2);
\draw[red,very thick](8) -- (4);
\draw(1) to[out = 170, in = 70] (8);
\draw({n-2}) to[out = 110, in = 85] (8);
\draw[red,very thick](5) to[out = 210, in = 290] (8);
\end{scope}
\end{tikzpicture}
\caption{Construction of a maximal planar, Hamiltonian graph on $n$ vertices, $n$ even.}
\end{figure}
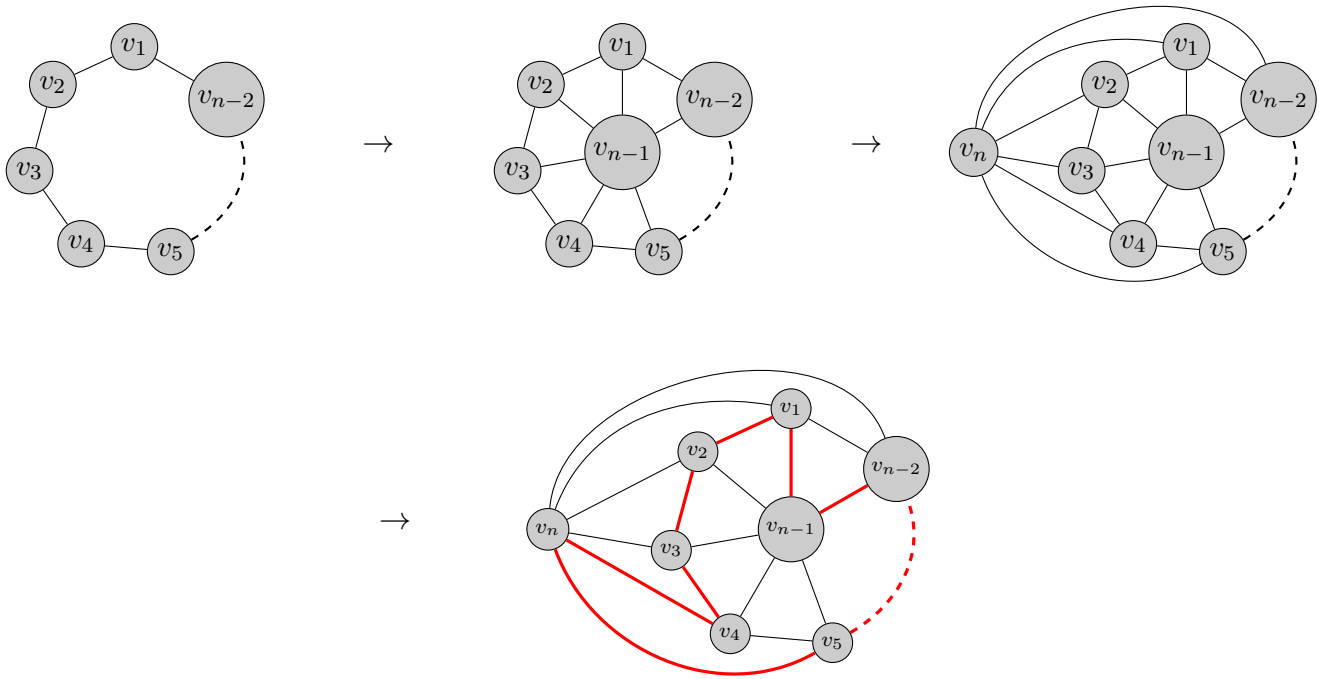

The resulting graph is Eulerian, as by construction all vertices have even degree. Further, it is Hamiltonian as can be seen in Figure 4. From Corollary 2 it follows that the graph is triangle decomposable. \\

\textbf{Case 2:} $n \geq 9$ and $n \equiv 1 \text{ (mod } 2)$:\\
}

A graph on $n$ vertices can be constructed as follows, see Figure 5: 
\begin{enumerate}
\item Form a $C_{n-2}$ graph. Label the vertices $v_1, v_2,...v_{n-2}$.
\item Connect $v_1$ and $v_3$ to form a triangle. 
\item Place an additional vertex $v_{n-1}$ on the inside of the cycle formed by $v_1$, $v_3$, $v_4$,...$v_{n-2}$ and draw edges between it and all vertices on the cycle.
\item Draw edges on the outside of the cycle connecting $v_{n-2}$ and $v_2$, $v_4$ and $v_2$, and $v_{n-2}$ and $v_4$. 
\item Place an additional vertex $v_n$ outside the outermost cycle of the graph. Connect it to all vertices on the outermost cycle, namely $v_4, v_5, v_6,...v_{n-2}$. 
\end{enumerate}

\begin{figure}[h]
\vspace{-2.75cm}
\hspace{-3cm}
\begin{tikzpicture}
\begin{scope}
[scale=.4,auto=left,every node/.style={circle,fill=gray!30},nodes={circle,draw, minimum size=.01cm, inner sep = 2pt}]
\foreach \lab/\ang in {1/90,2/135,3/180,4/225,5/270,6/315,{n-2}/35}
\node(\lab)at(\ang:4){\scriptsize{$v_{\lab}$}};
\draw({n-2})--(1) -- (2) -- (3) -- (4) -- (5) -- (6);
\draw[dashed, thick](6)to[out=60,in=290]({n-2});\end{scope}
\hspace{3cm}
$\rightarrow$
\hspace{3cm}
\begin{scope}
[scale=.4,auto=right,every node/.style={circle,fill=gray!40},nodes={circle,draw, minimum size=.01cm, inner sep = 2pt},every path/.style={}]
\foreach \lab/\ang in {1/90,2/135,3/180,4/225,5/270,6/315,{n-2}/35}
\node(\lab)at(\ang:4){\scriptsize{$v_{\lab}$}};
\draw({n-2})--(1) -- (2) -- (3) -- (4) -- (5) -- (6);
\draw(1) -- (3);
\draw[dashed, thick](6)to[out=60,in=290]({n-2});
\node(8)at(0,0){\scriptsize$v_{n-1}$};
\draw(1) -- (8) -- ({n-2});
\draw(3) -- (8) -- (4);
\draw(5) -- (8) -- (6);
\end{scope}
\hspace{2.5cm}
$\rightarrow$
\hspace{3cm}
\begin{scope}
[scale=.4,auto=right,every node/.style={circle,fill=gray!40},nodes={circle,draw, minimum size=.01cm, inner sep = 2pt},every path/.style={}]
\foreach \lab/\ang in {1/90,2/135,3/180,4/225,5/270,6/315,{n-2}/35}
\node(\lab)at(\ang:4){\scriptsize{$v_{\lab}$}};
\draw({n-2})--(1) -- (2) -- (3) -- (4) -- (5) -- (6);
\draw(1) -- (3);
\draw[dashed, thick](6)to[out=60,in=290]({n-2});
\node(8)at(0,0){\scriptsize$v_{n-1}$};
\draw(1) -- (8) -- ({n-2});
\draw(3) -- (8) -- (4);
\draw(5) -- (8) -- (6);
\draw({n-2}) to[out=110, in=70] (2);
\draw(4) to[in=200, out=150] (2);
\draw({n-2}) to[out=90, in=180,looseness=2.5] (4);
\end{scope}
\end{tikzpicture}
\end{figure}
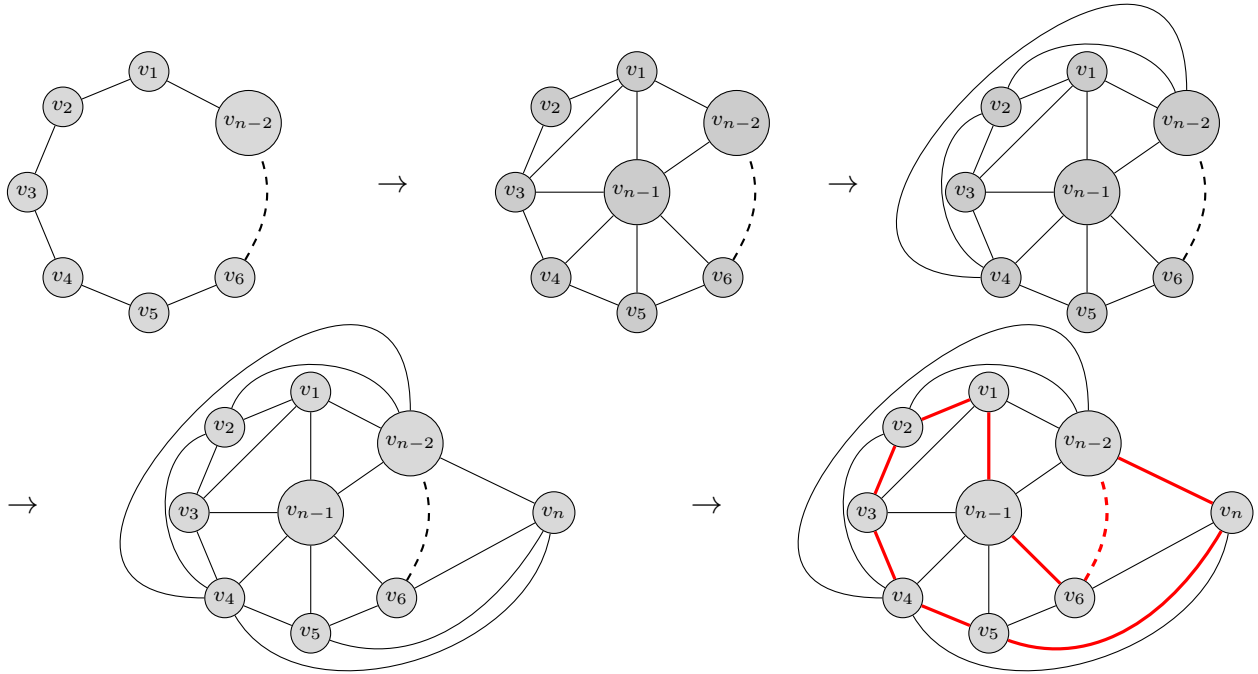
\begin{figure}[th!]
\centering
\vspace{-3cm}
\begin{tikzpicture}
\begin{scope}
\hspace{-9cm}
$\rightarrow$
\end{scope}
\hspace{-5cm}
\begin{scope}
[scale=.4,auto=left,every node/.style={circle,fill=gray!30},nodes={circle,draw, minimum size=.01cm, inner sep = 2pt}]
\foreach \lab/\ang in {1/90,2/135,3/180,4/225,5/270,6/315,{n-2}/35}
\node(\lab)at(\ang:4){\scriptsize{$v_{\lab}$}};
\node(8)at(0,0){\scriptsize$v_{n-1}$};
\node(9)at(8,0){\scriptsize$v_n$};
\draw({n-2})--(1);
\draw[](1)--(2)--(3)--(4)--(5);
\draw(5)--(6);
\draw[dashed, thick](6)to[out=60,in=290]({n-2});
\draw[](1)--(8);
\draw(1)--(3);
\draw(3)--(8);
\draw(4)--(8);
\draw(5)--(8);
\draw[](6)--(8);
\draw({n-2})--(8);
\draw({n-2}) to[out=110, in=70] (2);
\draw(4) to[in=200, out=150] (2);
\draw({n-2}) to[out=90, in=180,looseness=2.5] (4);
\draw[]({n-2})--(9);
\draw(6)--(9);
\draw[](5) to[out = 340, in =240](9);
\draw(4) to[out = 300, in =260](9);
\end{scope}
\hspace{5cm}
$\rightarrow$
\hspace{3.5cm}
\begin{scope}
[scale=.4,auto=left,every node/.style={circle,fill=gray!30},nodes={circle,draw, minimum size=.01cm, inner sep = 2pt}]
\foreach \lab/\ang in {1/90,2/135,3/180,4/225,5/270,6/315,{n-2}/35}
\node(\lab)at(\ang:4){\scriptsize{$v_{\lab}$}};
\node(8)at(0,0){\scriptsize$v_{n-1}$};
\node(9)at(8,0){\scriptsize$v_n$};
\draw({n-2})--(1);
\draw[very thick, red](1)--(2)--(3)--(4)--(5);
\draw(5)--(6);
\draw[dashed,very thick, red](6)to[out=60,in=290]({n-2});
\draw[very thick, red](1)--(8);
\draw(1)--(3);
\draw(3)--(8);
\draw(4)--(8);
\draw(5)--(8);
\draw[very thick, red](6)--(8);
\draw({n-2})--(8);
\draw({n-2}) to[out=110, in=70] (2);
\draw(4) to[in=200, out=150] (2);
\draw({n-2}) to[out=90, in=180,looseness=2.5] (4);
\draw[very thick, red]({n-2})--(9);
\draw(6)--(9);
\draw[very thick, red](5) to[out = 340, in =240](9);
\draw(4) to[out = 300, in =260](9);
\end{scope}
\end{tikzpicture}
\caption{Construction of a maximal planar, Hamiltonian graph on $n$ vertices, $n$ odd.}
\end{figure}

The resulting graph is Eulerian, as by construction all vertices have even degree. Further, it is Hamiltonian as can be seen in Figure 5. From Corollary 2 it follows that the graph is triangle decomposable. 
\end{proof}

Therefore we have proven the existence of triangle decomposable Hamiltonian maximal planar graphs with order $n \geq 3$ where $n \neq 4, 5, 7$. \\

\subsection{Simple-Clique 2-Trees}

Maximal planar graphs can be constructed recursively as simple-clique 2-trees. A \emph{k-tree} is a graph with $n \geq k$ vertices that is created by beginning with a $K_{k}$ and then adding additional vertices by connecting each new vertex to an existing $K_k$ in the graph. A \emph{simple-clique k-tree} is a $k$-tree such that when adding new vertices, each existing $K_k$ can be used at most once. We focus on simple-clique 2-trees, that is 2-trees with no edge on 3 or more triangles. \\

Simple-clique 2-trees and maximal outerplanar graphs are equivalent classes of graphs. Both maximal outerplanar graphs and simple-clique 2-trees have forbidden minors $K_4$ and $K_{2,3}$ \cite{sc3}. Therefore, our results in Section 3.1 follow for the class of simple-clique 2-trees. As simple-clique 2-trees are a subset of 2-trees, we have shown the existence of such graphs with minimal required multi-edges between already adjacent vertices for 2-trees. Let $2\mathcal{SC}_n$ be the set of all simple-clique 2-trees of order $n$. \\

Below, we provide an alternative constructive proof of Theorem 4 using the method of constructing simple-clique 2-trees. This proof method is simpler, as we need not consider congruence classes mod 4. \\

\textbf{Theorem 4. (V2)} \emph{For $n\geq3$, $\eps_\Delta(2\mathcal{SC}_n) = 0$ if and only if $n \equiv 0 \emph{\text{ (mod 3)}}$. }
\begin{proof}
First, suppose $\eps_\Delta(2\mathcal{SC}_n) = 0$. Then there exists $G \in 2\mathcal{SC}_n$ where $\eps_\Delta(G) = 0$. Therefore, $|E(G)| \equiv 0 \text{ (mod 3)}$. As $G$ is a simple-clique 2-tree, it follows that $|E(G)| = 2n - 3$. Therefore, $2n - 3 \equiv 0 \text{ (mod 3)}$. This implies that $n \equiv 0 \text{ (mod 3)}$, as desired. \\

Conversely, suppose $n \equiv 0 \text{ (mod 3)}$. A simple-clique 2-tree can be constructed on $n$ vertices as follows:
\begin{enumerate}
\item Begin by forming a $K_3$. This triangle is included in our triangle decomposition. 
\item Add a new vertex to the graph by connecting it to any allowable edge, that is any edge that has not already been used to add a vertex to the graph. This forms a new triangle with two edges that are not included in the triangle decomposition. 
\item Add two new vertices to the graph by connecting one to each of the edges constructed in step 2. This forms two new triangles, both of which are be included in the triangle decomposition. 
\item Repeat steps 2 and 3 until the desired number of vertices has been achieved.
\end{enumerate}

This construction method works as at each iteration, all newly added edges are included in the triangle decomposition. See Figure 6 for an example of order 9. \\
\end{proof}

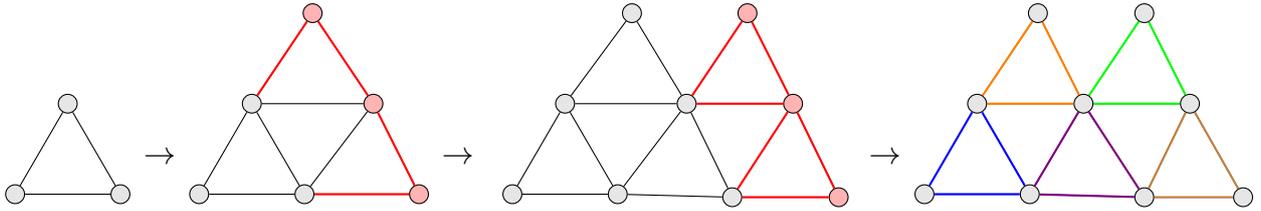
\begin{figure}[h]
\centering
\hspace{-12cm}
\begin{tikzpicture}\begin{scope}
[scale=.4,auto=right,every node/.style={circle,fill=gray!20},nodes={circle,draw,minimum size=0pt,inner sep=2.5pt}]
\foreach \lab/\ang in {a/90, b/210, c/330}
{\node(\lab)at(\ang:2){};}
\draw(a) -- (b) -- (c) -- (a);
\end{scope}
\hspace{1cm}
$\rightarrow$
\hspace{1cm}
\begin{scope}
[scale=.4,auto=right,every node/.style={circle,fill=gray!20},nodes={circle,draw,minimum size=0pt,inner sep=2.5pt}]
\foreach \lab/\ang in {a/90, b/210, c/330}
{\node(\lab)at(\ang:2){};}
\draw(a) -- (b) -- (c) -- (a);
\node[fill = red!30](d)at(4,2){};
\draw(a) -- (d) -- (c);
\node[fill = red!30](e)at(2,5){};
\node[fill = red!30](f)at(5.5,-1){};
\draw[red,thick](c)-- (f) -- (d);
\draw[red,thick](a)-- (e) -- (d);
\end{scope}
\hspace{2.5cm}
$\rightarrow$
\hspace{1.2cm}
\begin{scope}
[scale=.4,auto=right,every node/.style={circle,fill=gray!20},nodes={circle,draw,minimum size=0pt,inner sep=2.5pt}]
\foreach \lab/\ang in {a/90, b/210, c/330}
{\node(\lab)at(\ang:2){};}
\draw(a) -- (b) -- (c) -- (a);
\node[](d)at(4,2){};
\draw(a) -- (d) -- (c);
\node[](e)at(2.2,5){};
\node[](f)at(5.5,-1.1){};
\draw[](c)-- (f) -- (d);
\draw[](a)-- (e) -- (d);
\node[fill=red!30](g) at(6,5){};
\node[fill=red!30](h)at(9,-1.1){};
\node[fill=red!30](i)at(7.5,2){};
\draw[red,thick] (i)--(g)--(d)--(i)--(h)--(f)--(i);
\end{scope}
\hspace{4cm}
$\rightarrow$
\hspace{1cm}
\begin{scope}
[scale=.4,auto=right,every node/.style={circle,fill=gray!20},nodes={circle,draw,minimum size=0pt,inner sep=2.5pt}]
\foreach \lab/\ang in {a/90, b/210, c/330}
{\node(\lab)at(\ang:2){};}
\draw[blue,thick](a) -- (b) -- (c) -- (a);
\node[](d)at(3.5,2){};
\draw[](a) -- (d) -- (c);
\node[](e)at(2,5){};
\node[](f)at(5.5,-1.1){};
\draw[](c)-- (f) -- (d);
\draw[](a)-- (e) -- (d);
\node[](g) at(5.5,5){};
\node[](h)at(8.75,-1.1){};
\node[](i)at(7,2){};
\draw[] (i)--(g)--(d)--(i)--(h)--(f)--(i);
\draw[orange,thick](a) -- (e) -- (d) -- (a);
\draw[green,thick](d) -- (g) -- (i) -- (d);
\draw[violet,thick](d) -- (c) -- (f) -- (d);
\draw[brown,thick](i) -- (f) -- (h) -- (i);
\end{scope}
\end{tikzpicture}
\caption{A construction of an order 9 simple-clique 3-tree that is triangle decomposable, using the construction method described in Theorem 4 (V2).}
\end{figure}

This method provides alternate proofs of Theorem 5 and 6, beginning with the graphs in Figures 7 and 8. 

\begin{figure}[th!]
\centering
\begin{minipage}{.4\textwidth}
\centering
\begin{tikzpicture}
[scale=.7,auto=right,every node/.style={circle,fill=gray!20},nodes={circle,draw,minimum size=0pt,inner sep=2.5pt}]
\node(a)at(2,0){};
\node(b)at(0,-2){};
\node(c)at(-2,0){};
\node(d)at(0,2){};
\draw(d)--(a)--(b)--(c)--(d)--(b);
\draw[red, thick](b)to[bend right = 22](d);
  \end{tikzpicture}
  \captionof{figure}{An alternate proof for Theorem 5 begins with the above graph, where the red edge is a multi-edge. }
  \label{fig:test1}
  \end{minipage}
  \hspace{2cm}
  \begin{minipage}{.45\textwidth}
  \centering
  \begin{tikzpicture}
[scale=.7,auto=right,every node/.style={circle,fill=gray!20},nodes={circle,draw,minimum size=0pt,inner sep=2.5pt}]
\node(a)at(0,0){};
\node(b)at(0,3){};
\node(c)at(-3,0){};
\node(d)at(-1.5,-2.25){};
\node(e)at(2.25,1.5){};
\draw(c)--(a)--(b)--(c)--(d)--(a)--(e)--(b);
\draw[red, thick](c)to[bend right = 22](a);
\draw[red, thick](b)to[bend left = 22](a);
\end{tikzpicture}
\captionof{figure}{An alternate proof for Theorem 6 begins with the above graph.  }
\label{fig:test2}
\end{minipage}
  \end{figure}

\subsection{Simple-clique 3-trees}

After looking at simple-clique 2-trees, the natural progression is simple-clique 3-trees, as they both fit into the larger class of simple-clique $k$-trees.  The following lemma classifies all planar 3-trees. 
\begin{lemma} [Lemma 24 in \cite{sc3}] A \emph{3}-tree is maximal planar if and only if it is a simple-clique \emph{3}-tree. 
\end{lemma}

From Corollary 2, we know that maximal planar graphs are triangle decomposable if and if they are Eulerian. Notice that by construction a simple-clique 3-tree always has at least 2 odd degree vertices and so will always require at least three multi-edges between already adjacent vertices to form a triangle decomposable multigraph. The following theorem shows that for all $n \geq 4$, there is always a simple-clique 3-tree of order $n$ that requires the addition of exactly 3 multi-edges to be made triangle decomposable. 

\begin{theorem}
For any $n \geq 3$, there exists a simple-clique \emph{3}-tree $G$ such that $\eps_\Delta(G) = 3$. This is the minimum $\eps_\Delta(G)$ value for a simple clique \emph{3}-tree. 
\end{theorem}
\begin{proof}
The result is obvious if $n =3$. First, one can see easily that a simple-clique 3-tree cannot be Eulerian. This is because the final vertex added to the graph has degree 3, and so there must always be at least two odd degree vertices. Further, notice that the size of a simple clique 3-tree is $3n - 6$. Therefore, the minimum possible $\eps_\Delta$ value of a simple-clique 3-tree is 3. We will show that such a graph with order $n \geq 4$ exists for all $n$. \\

The following construction will produce a simple 3-clique tree $G$ such that $\eps_\Delta(G)$ = 3. See Figure 9. 
\begin{enumerate}
\item Consider an embedding of $K_4$ in the plance. Label the vertices $v_1,v_2,v_3,v_4$ such that $v_1$ is not on the outer face of the graph. Add an additional vertex $v$ on the outside of the graph so that it is adjacent to $v_2,v_3,v_4$. Set $i = 1$. 
\item Add a vertex within the face formed by $v, v_2, v_3$ and label it as $a_i$. 
\item Add a vertex within the face formed by $a_i, v_2, v_3$ and label it as $a_{i+1}$.
\item Repeat step 3 until the graph has order $n$, increasing $i$ by 1 each interation.
\item If $i$ is even (graph has odd order), add multi-edges parallel to $v_1v_2$, $v_2v_3$, $v_3a_i$. If $i$ is odd (graph has even order), add multi-edges parallel to $v_1v_2$, $v_2v_3$, $v_2a_i$. 
\end{enumerate}

This construction results in a triangle decomposable multigraph. \\

 If $i$ is even, we include all the triangles where one of the edges has multiplicity 2. These triangles are $v_1v_2v_4$, $v_1v_2v_3$, $v_2v_3a_i$, and $v_3a_ia_{i-1}$. Then, we include the triangles $a_{2j + 1}a_{2j}v_2$ and $a_{2j}a_{2j-1}v_3$ for all $1 \leq j \leq \frac{i-2}{2}$. Finally we include the triangles $a_1vv_2$ and $vv_4v_3$. \\
 
  If $i$ is odd, we include all the triangles where one of the edges has multiplicity 2. These triangles are $v_1v_2v_4$, $v_1v_2v_3$, $v_2v_3a_i$, and $v_2a_ia_{i-1}$. Then, we include the triangles $a_{2j}a_{2j - 1}v_3$ and $a_{2j-1}a_{2j-2}v_3$ for all $2 \leq j \leq \frac{i-1}{2}$. Finally we include the triangles $a_2a_1v_3$, $a_1v_2v$, and $vv_4v_3$. 

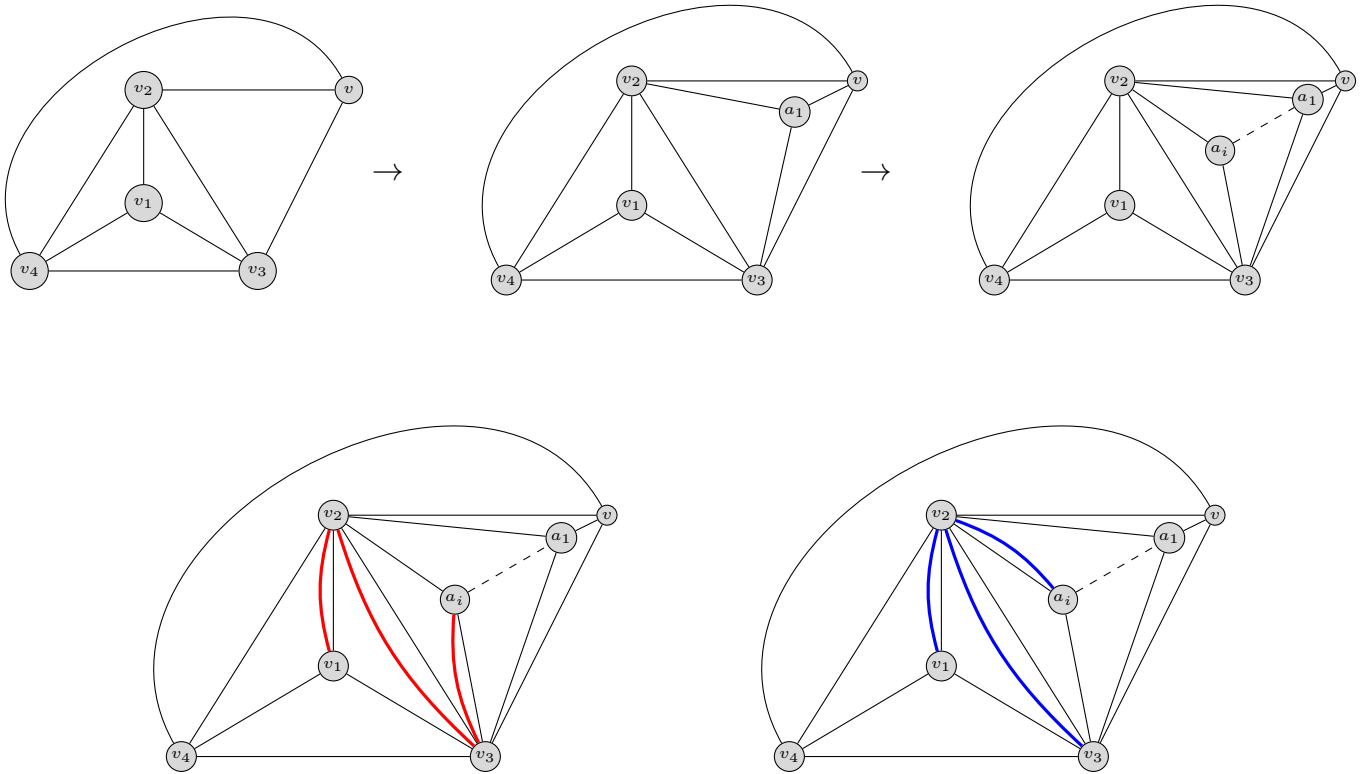
\begin{figure}[th!]
\centering
\begin{tikzpicture}
\hspace{-7cm}
\begin{scope}
[scale=.3,auto=right,every node/.style={circle,fill=gray!30, inner sep = 2pt,draw}]
\node(a) at(0,-1){\tiny$v_1$};
\node(b) at(0,4){\tiny$v_2$};
\node(c) at(5,-4){\tiny$v_3$};
\node(d) at(-5,-4){\tiny$v_4$};
\draw(a)--(b)--(c)--(a)--(d)--(c);
\draw(b)--(d);
\node(e) at(9,4){\tiny$v$};
\draw(b)--(e)--(c);
\draw(d) to[bend left = 90,looseness=1.25] (e);
\end{scope}
\hspace{3cm}
$\rightarrow$
\hspace{3cm}
\begin{scope}
[scale=.33,auto=right,every node/.style={circle,fill=gray!30, inner sep = 1pt,draw}]
\node(a) at(0,-1){\tiny$v_1$};
\node(b) at(0,4){\tiny$v_2$};
\node(c) at(5,-4){\tiny$v_3$};
\node(d) at(-5,-4){\tiny$v_4$};
\draw(a)--(b)--(c)--(a)--(d)--(c);
\draw(b)--(d);
\node(e) at(9,4){\tiny$v$};
\draw(b)--(e)--(c);
\draw(d) to[bend left = 90,looseness=1.25] (e);
\node(g)at(6.5,2.75){\tiny$a_1$};
\draw(b)--(g)--(c);
\draw(g)--(e);
\end{scope}
\hspace{3cm}
$\rightarrow$
\hspace{3cm}
\begin{scope}
[scale=.33,auto=right,every node/.style={circle,fill=gray!30, inner sep = 1pt,draw}]
\node(a) at(0,-1){\tiny$v_1$};
\node(b) at(0,4){\tiny$v_2$};
\node(c) at(5,-4){\tiny$v_3$};
\node(d) at(-5,-4){\tiny$v_4$};
\draw(a)--(b)--(c)--(a)--(d)--(c);
\draw(b)--(d);
\node(e) at(9,4){\tiny$v$};
\draw(b)--(e)--(c);
\draw(d) to[bend left = 90,looseness=1.25] (e);
\node(f)at(4,1.2){\tiny$a_i$};
\node(g)at(7.5,3.25){\tiny$a_1$};
\draw(b)--(f)--(c);
\draw[dashed](f)--(g);
\draw(b)--(g)--(c);
\draw(g)--(e);
\end{scope}
\end{tikzpicture}\\
\begin{tikzpicture}
\hspace{-7.5cm}
\hspace{3cm}
\begin{scope}
[scale=.4,auto=right,every node/.style={circle,fill=gray!30, inner sep = 1pt,draw}]
\node(a) at(0,-1){\tiny$v_1$};
\node(b) at(0,4){\tiny$v_2$};
\node(c) at(5,-4){\tiny$v_3$};
\node(d) at(-5,-4){\tiny$v_4$};
\draw(a)--(b)--(c)--(a)--(d)--(c);
\draw(b)--(d);
\node(e) at(9,4){\tiny$v$};
\draw(b)--(e)--(c);
\draw(d) to[bend left = 90,looseness=1.25] (e);
\node(f)at(4,1.2){\tiny$a_i$};
\node(g)at(7.5,3.25){\tiny$a_1$};
\draw(b)--(f)--(c);
\draw[dashed](f)--(g);
\draw(b)--(g)--(c);
\draw(g)--(e);
\draw[red,very thick](a) to[bend left =15] (b);
\draw[red,very thick](c) to[bend left =15] (b);
\draw[red,very thick](c) to[bend left =15] (f);
\end{scope}
\hspace{5cm}
\hspace{3cm}
\begin{scope}
[scale=.4,auto=right,every node/.style={circle,fill=gray!30, inner sep = 1pt,draw}]
\node(a) at(0,-1){\tiny$v_1$};
\node(b) at(0,4){\tiny$v_2$};
\node(c) at(5,-4){\tiny$v_3$};
\node(d) at(-5,-4){\tiny$v_4$};
\draw(a)--(b)--(c)--(a)--(d)--(c);
\draw(b)--(d);
\node(e) at(9,4){\tiny$v$};
\draw(b)--(e)--(c);
\draw(d) to[bend left = 90,looseness=1.25] (e);
\node(f)at(4,1.2){\tiny$a_i$};
\node(g)at(7.5,3.25){\tiny$a_1$};
\draw(b)--(f)--(c);
\draw[dashed](f)--(g);
\draw(b)--(g)--(c);
\draw(g)--(e);
\draw[blue,very thick](a) to[bend left =15] (b);
\draw[blue,very thick](c) to[bend left =15] (b);
\draw[blue,very thick](b) to[bend left =15] (f);
\end{scope}
\end{tikzpicture}
\caption{Construction of a simple-clique 3-tree with the minimum number of multi-edges required to form a triangle decomposable graph. For odd order, multi-edges are shown in red; for even order, multi-edges are shown in blue.}
\end{figure}

\end{proof}

\section{Toroidal Graphs}

After exploring multiple classes of planar graphs, it is natural to look at triangulations of the torus. There is a larger number of graphs embeddable on the torus than in the plane; for example $K_5$ is embeddable in the torus while it is famously non-planar. See Figure 10. \\

\begin{figure}[h]
\centering
\includegraphics[scale=.2]{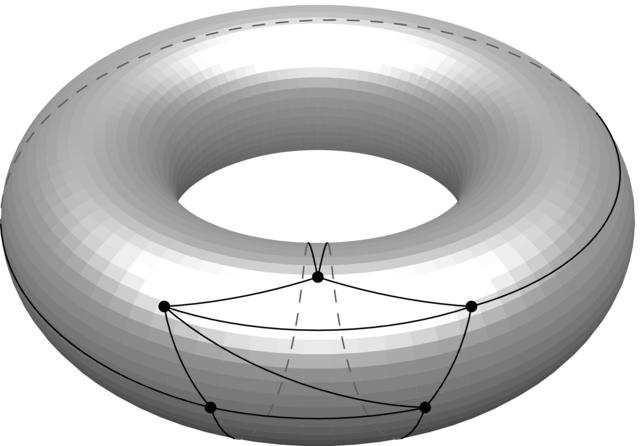}
\hspace{3cm}
\begin{tikzpicture}
[scale = 1.3]
\draw[black, thick] (0,1) -- (1,0.25) -- (0.6,-1) -- (-.6,-1) -- (-1,.25) -- (0,1);
\draw[black, thick] (1,0.25) -- (-1,0.25) -- (0.6,-1);
 \draw[black, very thick, postaction={decoration={markings,mark=at position 0.5 with {\arrow{>>}}},decorate}] (1.5,1.5) -- (-1.5,1.5); 
  \draw[black, very thick, postaction={decoration={markings,mark=at position 0.5 with {\arrow{<<}}},decorate}] (-1.5,-1.5) -- (1.5,-1.5); 
   \draw[black, very thick, postaction={decoration={markings,mark=at position 0.5 with {\arrow{>}}},decorate}] (-1.5,-1.5) -- (-1.5,1.5); 
  \draw[black, very thick, postaction={decoration={markings,mark=at position 0.5 with {\arrow{<}}},decorate}] (1.5,1.5) -- (1.5,-1.5); 
 \draw[red, thick](0,1) -- (.3, 1.5);
  \draw[red, thick](.6,-1) -- (.3, -1.5);
   \draw[blue, thick](0,1) -- (-.3, 1.5);
  \draw[blue, thick](-.6,-1) -- (-.3, -1.5);
  \draw[green, thick] (1,.25) -- (1.5,-.2);
    \draw[green, thick] (-.6,-1) -- (-1.5,-.2);
\def\r{.1} \draw[fill=red!40]  (0,1) circle(\r) (1,0.25) circle(\r) (-1,0.25) circle(\r) (0.6,-1) circle(\r) (-0.6,-1) circle(\r); \end{tikzpicture}
\caption{Two different representations of $K_5$ embedded on the torus (left is a rendering by Anthony Phan, see \cite{phan}).}
\end{figure}
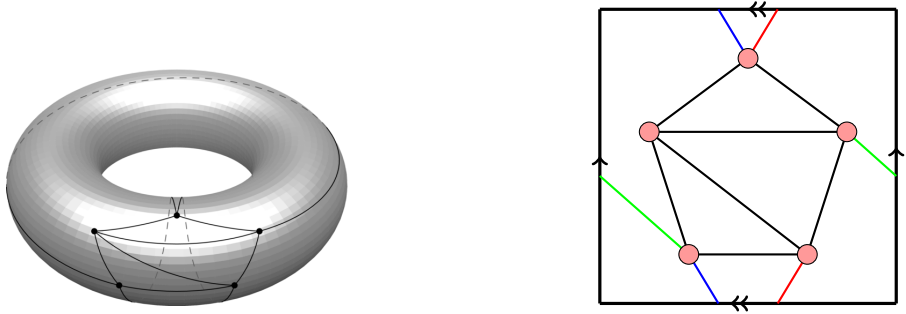

Euler's characterization for polyhedra can be extended to planar graphs to give a relationship between the number of edges, faces, and vertices of a given graph. It states that for any graph embedded in the plane $V - E + F = 2,$ where $V$ is the order of the graph, $E$ is the size, and $F$ is the number of faces. This formula was generalized by Simon Lhuilier for polyhedra with holes and thus could be used for graphs embedded on torus-like polyhedra. The generalization states that $V - E + F = 2 - 2G,$ where $V$, $E$, and $F$ are as in Euler's characterization and $G$ is the `genus' of the polyhedra, which is the number of holes. Therefore, for the torus we know that $V - E + F = 0.$ \\

It is interesting to note that embedding a maximal planar graph in the plane results in a a triangulation of the plane but this does not extend to the torus. There exist graphs of maximal size on the torus that have faces which are not triangles. See Figure 11. \\
\begin{figure}[h]
\centering
\begin{tikzpicture}
[scale=1.3]

\draw[very thick] (2,1) -- (0,1) -- (-2,1) -- (-2,0) -- (-2,-1) -- (0,-1) -- (2, -1) -- (2,0) -- (2,1);
\draw [very thick](2,0) -- (0,-1) -- (-2,0) -- (0,1) -- (2,0);
\draw[very thick, red] (0,1) -- (0,2);
\draw[very thick, red] (0,-1) -- (0,-2);
\draw[very thick, green] (2,0) -- (3,0);
\draw[very thick, green] (-2,0) -- (-3,0);
\draw[very thick, blue] (-2,-1) -- (-3,-1);
\draw[very thick, blue] (2,-1) -- (3,-1);
\draw[very thick, violet] (-2,1) -- (-3,1);
\draw[very thick, violet] (2,1) -- (3,1);
\draw[very thick, cyan] (-2,1) -- (-2,2);
\draw[very thick, cyan] (-2,-1) -- (-2,-2);
\draw[very thick, pink] (2,1) -- (2,2);
\draw[very thick, pink] (2,-1) -- (2,-2);
\draw[very thick, orange] (-2,-1) -- (-1,-2);
\draw[very thick, orange] (0,1) -- (-1,2);
\draw[very thick, gray] (2,-1) -- (1,-2);
\draw[very thick, gray] (0,1) -- (1,2);
\draw[very thick, brown] (-2,1) -- (-3,2);
\draw[very thick, brown] (2,-1) -- (3,-2);
\draw[very thick, purple] (-2,0) -- (-3,0.5);
\draw[very thick, purple] (2,1) -- (3,0.5);
\draw[very thick, olive] (-2,0) -- (-3,-0.5);
\draw[very thick, olive] (2,-1) -- (3,-0.5);
\def\r{.075} \draw[fill=gray!40]   (2,1) circle(\r) (0,1) circle(\r) (-2,1) circle(\r) (2,0) circle(\r) (-2,0) circle(\r) (2,-1) circle(\r) (-2,-1) circle(\r) (0,-1) circle(\r);
 \draw[black,very thick] (3,2) rectangle (-3,-2); 
  \draw[black, very thick, postaction={decoration={markings,mark=at position 0.6 with {\arrow{>>}}},decorate}] (3,2) -- (-3,2); 
  \draw[black, very thick, postaction={decoration={markings,mark=at position 0.4 with {\arrow{<<}}},decorate}] (-3,-2) -- (3,-2); 
   \draw[black, very thick, postaction={decoration={markings,mark=at position 0.4 with {\arrow{<}}},decorate}] (-3,-2) -- (-3,2); 
  \draw[black, very thick, postaction={decoration={markings,mark=at position 0.6 with {\arrow{>}}},decorate}] (3,2) -- (3,-2); 
\end{tikzpicture}
\caption{A maximal toroidal graph that is not a triangulation of the torus \cite{white}. }
\end{figure}
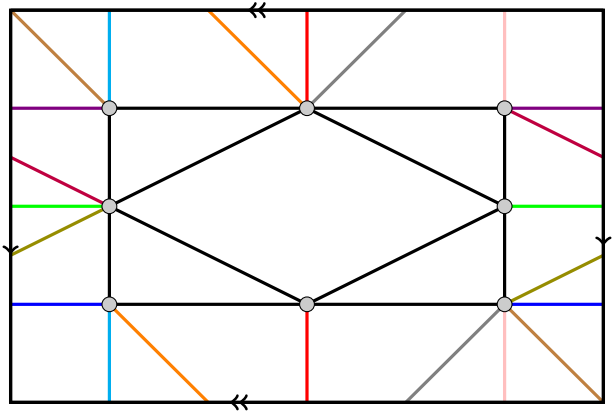

We define a \emph{maximal single-face toroidal graph} as a graph that is embedded on the torus, all vertices lie on one face that is a cycle, all other edges lie outside the cycle such that the addition of any edge between not already adjacent vertices would force edges to cross, and all edges lie on at least one triangle. We use $\mathcal{SF}_n$ to represent the set of all single-faced toroidal graphs with order $n$. \\
\begin{lemma}
A graph $G \in \mathcal{SF}_n$ has size $2n + 3$. 
\end{lemma}
\begin{proof}
Let $G \in \mathcal{SF}_n$. Then, $G$ has $n$ edges on one face that forms a cycle. All edges on the cycle lie on one triangle and no edges lie inside the triangle, from the definition of maximal single-face toroidal graphs. Therefore, $|E(G)| = \frac{n + 3(f-1)}{2}$. It follows from Euler's formula for toroidal graphs that $|E(G)| = 2n + 3$. 
\end{proof}

As the size of a graph in $\mathcal{SF}_n$ is $2n+3$, the minimum possible number of edges required to form a strongly $K_3$-divisible multigraph is the residue of $n \text{ (mod 3)}$. We found a construction of a graph in $\mathcal{SF}_8$ that requires two multi-edges, and so it follows that $\eps_\Delta(\mathcal{SF}_8) = 2$. See Figure 12. \\

Similarly, the number of multi-edges required of a graph in $\mathcal{SF}_7$ must be in the congruence class $1 \text{ (mod 3)}$. We found a construction of a graph that requires six multi-edges, and so $\eps_\Delta(\mathcal{SF}_7) \leq 4$. See Figure 13. \\

Finally, the number of multi-edges required of a graph in $\mathcal{SF}_9$ must be in the congruence class $0 \text{ (mod 3)}$. We found a construction of a graph that requires four multi-edges, and so $\eps_\Delta(\mathcal{SF}_9) \leq 6$. See Figure 14. \\

\begin{figure}[th!]
\centering
\begin{minipage}{.4\textwidth}
\centering
\begin{tikzpicture}
[scale=1.8]
\draw[black, thick](1,0)--(0.5,1)--(-.5,1) -- (-1,0)--(-1.5,-1)--(-.5,-1)--(.5,-1) -- (1,0) -- (0,0)--(.5,-1);
\draw[violet, thick](-1.5,-1) -- (-2,-1.3);
\draw[violet, thick](1.5,-1.3) -- (1.2,-1.5);
\draw[violet, thick](1.2,1.5) -- (0.5,1);
\draw[blue, thick](-1.5,-1) -- (-2, -1);
\draw[blue, thick](1.5,-1) -- (.5, -1);
\draw[cyan, thick](-1.5,-1)--(-1,-1.5);
\draw[cyan, thick](-1,1.5)--(-.5,1);
\draw[orange, thick](-1.5,-1) -- (-1.7,-1.5);
\draw[orange, thick] (-1.7,1.5) -- (-2,1);
\draw[orange, thick] (1.5,1) -- (1,0);
\draw[pink, double, very thick] (-.5,-1) -- (-.5,-1.5);
\draw[pink, double, very thick] (-.5,1.5) -- (-.5,1);
\draw[red, double, very thick](-.5,-1) -- (0,-1.5);
\draw[red, double, very thick](0,1.5) -- (.5,1);
\draw[yellow, thick](.5,1) -- (.5,1.5);
\draw[yellow, thick](.5,-1.5) -- (.5,-1);
\draw[brown, thick](.5,-1) -- (1.5,-.5);
\draw[brown, thick](-2,-.5) -- (-1,0);
\draw[green, thick](-1,0) -- (-2,0);
\draw[green, thick](1.5,0) -- (1,0);
\draw[purple, thick](-.5,1) -- (-2,.25);
\draw[purple, thick](1.5,.25) -- (1,0);
\def\r{.05} \draw[fill=gray!40]   (-1.5,-1) circle(\r) (.5,-1) circle(\r) (-1,0) circle(\r) (-0.5,1) circle(\r) (0.5,1) circle(\r) (-.5,-1) circle(\r) (0,0) circle(\r) (1,0) circle(\r); 
  \draw[black, very thick, postaction={decoration={markings,mark=at position 0.5 with {\arrow{>>}}},decorate}] (1.5,1.5) -- (-2,1.5); 
  \draw[black, very thick, postaction={decoration={markings,mark=at position 0.5 with {\arrow{<<}}},decorate}] (-2,-1.5) -- (1.5,-1.5); 
   \draw[black, very thick, postaction={decoration={markings,mark=at position 0.45 with {\arrow{>}}},decorate}] (-2,-1.5) -- (-2,1.5); 
  \draw[black, very thick, postaction={decoration={markings,mark=at position 0.55 with {\arrow{<}}},decorate}] (1.5,1.5) -- (1.5,-1.5); 
  \end{tikzpicture}
  \captionof{figure}{An example that shows $\eps_\Delta(\mathcal{SF}_8) = 2$. }
  \label{fig:test1}
  \end{minipage}
  \hspace{2cm}
  \begin{minipage}{.45\textwidth}
  \centering
  \begin{tikzpicture}
[scale=1.5]
\draw[black,thick] (0,1.3) -- (1,.7) -- (0,0) -- (0,1.3) -- (-1,.7) -- (-1,-.7) -- (0,-1.3) -- (1,-.7) -- (1,.7);
\draw[violet, very thick, double](0,1.3) -- (0,2);
\draw[violet, very thick, double](0,-2) -- (0,-1.3);
\draw[magenta, thick](0,-1.3) -- (-.4,-2);
\draw[magenta, thick](-.4,2) -- (-1,.7);
\draw[cyan, thick](0,-1.3) -- (-1,-2);
\draw[cyan, thick](-1,2) -- (-2,1.4);
\draw[cyan, thick](2,1.4) -- (1,.7);
\draw[orange,very thick, double] (0,1.3) -- (.5,2);
\draw[orange, very thick, double] (.5,-2) -- (1,-.7);
\draw[red,very thick, double] (1,-.7) -- (2,-.7);
\draw[red, very thick, double] (-2,-.7) -- (-1,-.7);
\draw[brown, thick] (1,.7) -- (2,.7);
\draw[brown, thick] (-1,.7) -- (-2,.7);
\draw[blue, very thick, double] (-1,.7) -- (-2,.1);
\draw[blue, very thick, double] (2,.1) -- (1,-.7);
\draw[green, thick](1,.7) -- (2,2);
\draw[green, thick](-1,-.7) -- (-2,-2);
\draw[olive, thick](-1,-.7) -- (-2,-1.3);
\draw[olive, thick](2,-1.3) -- (1.2,-2);
\draw[olive, thick](1.2,2) -- (0,1.3);
\def\r{.05} \draw[fill=gray!40]   (0,0) circle(\r) (0,1.3) circle(\r) (-1,.7) circle(\r) (1,.7) circle(\r) (-1,-.7) circle(\r) (1,-.7) circle(\r) (0,-1.3) circle(\r);
  \draw[black, very thick, postaction={decoration={markings,mark=at position 0.5 with {\arrow{>>}}},decorate}] (2,2) -- (-2,2); 
  \draw[black, very thick, postaction={decoration={markings,mark=at position 0.5 with {\arrow{<<}}},decorate}] (-2,-2) -- (2,-2); 
   \draw[black, very thick, postaction={decoration={markings,mark=at position 0.5 with {\arrow{>}}},decorate}] (-2,-2) -- (-2,2); 
  \draw[black, very thick, postaction={decoration={markings,mark=at position 0.5 with {\arrow{<}}},decorate}] (2,2) -- (2,-2); 
\end{tikzpicture}
\captionof{figure}{An example that shows $\eps_\Delta(\mathcal{SF}_7) \leq 4$. }
\label{fig:test2}
\end{minipage}
  \end{figure}
\begin{figure}[th!]
\centering
\begin{tikzpicture}
[scale=1.3]
\draw[black, thick](-2,1) -- (0,1) -- (2,1) -- (3,0) -- (2,-1) -- (0,-1) -- (-2,-1) -- (-3,0) --(-2,1);
\draw[black, thick,double](-1,0) -- (-3,0) --(-2,1);
\draw[black,thick](-2,1) -- (-1,0) -- (-2,-1);
\draw[black,thick](2,-1) -- (1,0) -- (3,0);
\draw[green,thick](2,-1) -- (3,-1);
\draw[green,thick](-2,-1) -- (-3,-1);
\draw[magenta,thick](2,1) -- (3,1);
\draw[magenta,thick](-2,1) -- (-3,1);
\draw[orange,thick](-2,1) -- (-3,2);
\draw[orange,thick](2,-1) -- (3,-2);
\draw[cyan,thick](2,1) -- (2,2);
\draw[cyan,thick](2,-1) -- (2,-2);
\draw[red,thick,double](-2,1) -- (-2,2);
\draw[red,thick,double](-2,-1) -- (-2,-2);
\draw[blue,thick,double](0,1) -- (0,2);
\draw[blue,thick,double](0,-1) -- (0,-2);
\draw[violet,thick,double](-2,-1) -- (-1,-2);
\draw[violet,thick,double](0,1) -- (-1,2);
\draw[olive,thick,double](0,-1) -- (1,-2);
\draw[olive,thick,double](2,1) -- (1,2);
\def\r{.075} \draw[fill=gray!40]   (1,0) circle(\r) (-1,0) circle(\r) (3,0) circle(\r) (-3,0) circle(\r)
 (-2,1) circle(\r)  (2,1) circle(\r)  (0,1) circle(\r)  (-2,-1) circle(\r)  (2,-1) circle(\r)  (0,-1) circle(\r);
  \draw[fill=white,white,very thick] (3,-1) rectangle (4,2); 
    \draw[fill=white,white,very thick] (-3,-1) rectangle (-4,2); 
 \draw[black,very thick] (3,2) rectangle (-3,-2); 
  \draw[black, very thick, postaction={decoration={markings,mark=at position 0.6 with {\arrow{>>}}},decorate}] (3,2) -- (-3,2); 
  \draw[black, very thick, postaction={decoration={markings,mark=at position 0.4 with {\arrow{<<}}},decorate}] (-3,-2) -- (3,-2); 
   \draw[black, very thick, postaction={decoration={markings,mark=at position 0.4 with {\arrow{<}}},decorate}] (-3,-2) -- (-3,2); 
  \draw[black, very thick, postaction={decoration={markings,mark=at position 0.6 with {\arrow{>}}},decorate}] (3,2) -- (3,-2); 
\end{tikzpicture}
\caption{An example that shows $\eps_\Delta(\mathcal{SF}_9) \leq 6$. }
\end{figure}
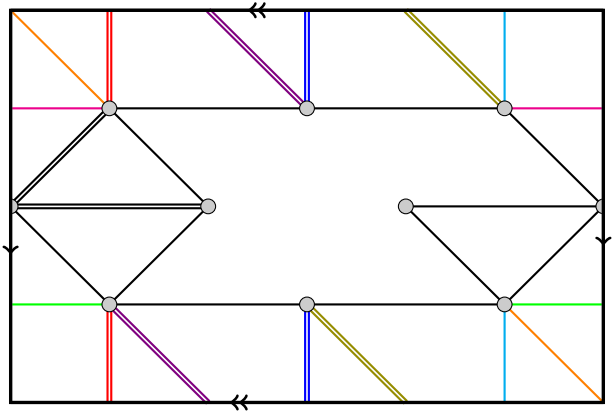

The problem of determining $\eps_\Delta(\mathcal{SF}_n)$ for general $n$ remains open. \\

\section{Concluding remarks and open problems}

In this paper, we have shown, using constructive methods, the existence of triangle decomposable multigraphs with a minimum number of multi-edges. Our focus has mainly been on graphs with the minimum number of multi-edges, but the existence of graphs with non-minimum $\eps_\Delta$ values remains open for most of the classes we explored. \\

Further, we did not prove the existence of single-face toroidal graphs with minimum number of multi-edges for $n \neq 8$. The existence of such graphs remains as an open problem. \\

Finally, the existence of triangle decomposable graphs with minimum number of multi-edges in other classes of graphs, especially non-planar classes, remains open. We explored simple-clique $k$-trees when $k = 2, 3$; the case for general $k$ remains open. To explore this problem, it may be necessary to understand the structure of adding multi-edges to $K_n$ to form a strongly $K_3$-divisible graph that is triangle decomposable. When $n \equiv 1, 3 \mod 6$, by existence of Steiner Triple systems, $K_n$ is triangle decomposable. For all other $n$, the problem was solved by Fort and Hedlund in 1958 \cite{fort}. They proved that adding the minimum possible number of edges required to form a strongly $K_3$-divisible graph to $K_n$ results in a triangle decomposable multigraph. \\

The cases for simple-clique $k$-trees may also require the step number to be considered, as with the addition of each new vertex there is a minimum number of odd vertices required. For example, a simple-clique 4-tree with five vertices has no odd degree vertices, a simple-clique 4-tree with six vertices contains exactly 4 odd vertices, and a simple-clique 4-tree with seven vertices contains at least two odd degree vertices. This gives a lower bound on the required number of multi-edges. \\ 

\textbf{Acknowledgements.} We would like to thank Gary MacGillivary, Peter Dukes, and Therese Biedl for their assistance. \\

We acknowledge the support of the Natural Sciences and Engineering Research Council of Canada (NSERC), PIN 253271. \\

Cette recherche a \'{e}t\'{e} financ\'{e}e par le Conseil de recherches en sciences naturelles et en g\'{e}nie du Canada (CRSNG), PIN 253271.

\begin{figure}[th!]
\begin{center}
\includegraphics[scale=.25]{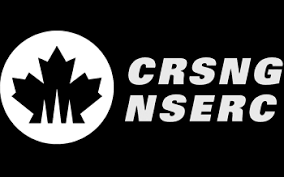}
\end{center}
\end{figure}

\newpage

\section{Appendix}
\subsection{Additional proofs}
\textbf{Theorem 6.} \emph{ For $n\geq 3$, $n \equiv 1 \mod 3$ if and only if $\eps_{\Delta}(\mathcal{OP}_{n}) = 1$.}
\begin{proof}

Suppose first that $\eps_{\Delta}(\mathcal{OP}_{n}) = 1$. Then, there exists $G \in \mathcal{OP}_n$ such that $\eps_{\Delta}(G) = 1$. Let us fix such a $G$. Then $E(G) + 1 \equiv 0 \text{ (mod 3)}$. Therefore, $2n -3 + 1 \equiv 0 \text{ (mod 3)}$. This implies $n \equiv 1 \text{ (mod 3)}$, as desired.  \\

Conversely, suppose $n \equiv 1 \text{ (mod 3)}$. We will show there exists some graph $G \in \mathcal{OP}_n$ where $\eps_\Delta(G) = 1$. It is convenient to look at four cases, one for each congruence class mod 4. We proceed by induction on $n$. \\

First, consider the following graphs on $n$ vertices, where $n = 4,7, 10, 13$. The order of each of these graphs represents a different congruence class mod 4. 

$$n = 4: \\
\begin{tikzpicture}
[scale=.75,auto=right,every node/.style={circle,fill=gray!30, draw}]
\foreach \lab/\ang in {a/45,b/135,c/-135,d/-45}
{\node(\lab) at(\ang:2){};}
\draw[very thick,red]  (a) -- (b) -- (c) ;
\draw[very thick,blue] (c)-- (d) -- (a);
\draw[very thick,red] (a) to[in = 75, out= 195] (c);
\draw[very thick,blue] (a) to[in = 15, out= 255] (c);
\end{tikzpicture} \hspace{5mm}
n = 7: \\
\begin{tikzpicture}
[scale=.75,auto=right,every node/.style={circle,fill=gray!30, draw}]
\foreach \lab/\ang in {a/90,b/141.43,c/192.86,d/244.29,e/295.72,f/347.15,g/398.58}
{\node(\lab) at(\ang:2){};}
\draw[very thick,red](a)--(b)--(c)--(a);
\draw[very thick,blue](c)--(d)--(e)--(c);
\draw[very thick, green](e)--(f)--(g)--(e);
\draw[very thick, orange](g)--(a)--(e);
\draw[very thick, orange](g)to[in=95,out=240](e);
\end{tikzpicture}$$

$$n = 10: \\
\begin{tikzpicture}
[scale=.75,auto=right,every node/.style={circle,fill=gray!30, draw}]
\foreach \lab/\ang in {a/90,b/126,c/162,d/198,e/234,f/270,g/304,h/342,i/378,j/414,k/450}
{\node(\lab) at(\ang:2){};}
\draw[very thick,red](a)--(b)--(c)--(a);
\draw[very thick,blue](c)--(d)--(e)--(c);
\draw[very thick, green](e)--(f)--(g)--(e);
\draw[very thick, orange](g)--(h)--(i)--(g);
\draw[very thick, cyan](i)--(j)--(k)--(i);
\draw[very thick, pink](a) -- (e) -- (i);
\draw[very thick, pink](a) to[out = 300, in=180] (i);
\end{tikzpicture}\hspace{5mm}
n = 13: \\
\begin{tikzpicture}
[scale=1,auto=right,every node/.style={circle,fill=gray!30, draw}]
\foreach \lab/\ang in {a/0,b/27.7,c/55.4,d/83.1,e/110.8,f/138.5,g/166.2,h/193.9,i/221.6,j/249.3,k/277,l/304.7,m/332.4}
{\node(\lab) at(\ang:3){};}
\draw[very thick,red](a)--(b)--(c)--(a);
\draw[very thick,blue](c)--(d)--(e)--(c);
\draw[very thick, green](e)--(f)--(g)--(e);
\draw[very thick, orange](g)--(h)--(i)--(g);
\draw[very thick, cyan](i)--(j)--(k)--(i);
\draw[very thick, brown](k)--(l)--(m)--(k);
\draw[very thick, pink](a)--(m)--(e)--(a);
\draw[very thick, purple](m)--(g)--(k);
\draw[very thick, purple](k) to[out=90,in=180] (m);
\end{tikzpicture}$$ \\

Next, assume we there exists a maximal outerplanar graph $G$ of order $3(k-1) + 1$ where $\eps_\Delta(G) = 1$, for some $k \geq 2$. Recall that all maximal outerplanar graphs are Hamiltonian, so this assumption is equivalent to saying the inside of any cycle of length $3(k-1) + 1$ can be triangulated to form a graph $G$ where $\eps_\Delta(G) = 1$. \\

For the induction step, we will consider four cases, one for each congruence class mod 4. \\

\emph{Case 1:} Suppose $k = 4r$ for for $r \in \N$ where $r < k$. Then $n = 12r + 1$. \\

Begin by forming a cycle $C_{12r+1}$. Connect every other vertex in the graph beginning at $v_1$ until $v_{12r+1}$. Every edge other than $v_1v_{12r+1}$ lies in a triangle. Form a triangle between $v_{1}$, $v_{12r+9}$ and $v_{12r-3}$ by adding an edge between $v_1$ and $v_{12r-3}$ and an edge between $v_{12r+1}$ and $v_{12r-3}$. The inner-cycle of the graph has order $6r-3$. 

$$\begin{tikzpicture}
[scale=.6,auto=right,every node/.style={circle,fill=gray!30, inner sep = 1pt,draw}]
\foreach \labb/\lab /\ang in {1/$v_1$/90,2/$v_2$/75,3/$v_3$/60,4/$v_4$/45,5/$v_5$/30,6/$v_6$/15,
7/$v_{12r+1}$/105}
\node[](\labb)at(\ang:7){\tiny\lab};
\foreach \labb/\lab /\ang in {11/$v_{12r-3}$/165,12/$v_{12r-4}$/180,13/$v_{12r-5}$/195,14/$v_{12r-6}$/210,15/$v_{12r-7}$/225,10/$v_{12r-2}$/150,9/$v_{12r-1}$/135,8/$v_{12r}$/120}
\node[](\labb)at(\ang:7){\tiny\lab};
\foreach \labb/\lab /\ang in {16/$v_{12r-8}$/240,17/$v_{12r-9}$/255}
\node[](\labb)at(\ang:7){\tiny\lab};
\node[](19)at(0:7){\tiny$v_7$};
\draw(17)--(16)--(15)--(14)--(13)--(12)--(11)--(10)--(9)--(8)--(7)--(1)--(2)--(3)--(4)--(5)--(6)--(19);
\draw[red,very thick](1)to[bend right = 22](3);
\draw[red,very thick](3)to[bend right = 22](5);
\draw[red,very thick](5)to[bend right = 22](19);
\draw[red,very thick](17)to[bend right = 22](15);
\draw[red,very thick](15)to[bend right = 22](13);
\draw[red,very thick](13)to[bend right = 22](11);
\draw[red,very thick](11)to[bend right = 22](9);
\draw[red,very thick](9)to[bend right = 22](7);
\draw[red,very thick](7)to[bend left = 22](11);
\draw[red,very thick](11)to[bend right = 22](1);
\node[fill=none,draw=none,draw = none](a)at(-10:7){};
\node[fill=none,draw=none,draw = none](b)at(-10:6.5){};
\draw[dashed](19) -- (a);
\draw[dashed,red,very thick](19) -- (b);
\node[fill=none,draw=none,draw = none](a')at(265:7){};
\node[fill=none,draw=none,draw = none](b')at(265:6.5){};
\draw[dashed](17) -- (a');
\draw[dashed,red,very thick](17) -- (b');
\end{tikzpicture}$$

Form a triangle with vertices $v_1$,  $v_{12r-5}$ and $v_{12r-9}$. This new triangle is included in our triangle decomposition. The innermost cycle is of order $6r - 4$. \\

$$\begin{tikzpicture}
[scale=.6,auto=right,every node/.style={circle,fill=gray!30, inner sep = 1pt,draw}]
\foreach \labb/\lab /\ang in {1/$v_1$/90,2/$v_2$/75,3/$v_3$/60,4/$v_4$/45,5/$v_5$/30,6/$v_6$/15,
7/$v_{12r+1}$/105}
\node[](\labb)at(\ang:7){\tiny\lab};
\foreach \labb/\lab /\ang in {11/$v_{12r-3}$/165,12/$v_{12r-4}$/180,13/$v_{12r-5}$/195,14/$v_{12r-6}$/210,15/$v_{12r-7}$/225,10/$v_{12r-2}$/150,9/$v_{12r-1}$/135,8/$v_{12r}$/120}
\node[](\labb)at(\ang:7){\tiny\lab};
\foreach \labb/\lab /\ang in {16/$v_{12r-8}$/240,17/$v_{12r-9}$/255}
\node[](\labb)at(\ang:7){\tiny\lab};
\node[](19)at(0:7){\tiny$v_7$};
\draw(17)--(16)--(15)--(14)--(13)--(12)--(11)--(10)--(9)--(8)--(7)--(1)--(2)--(3)--(4)--(5)--(6)--(19);
\draw[red,very thick](1)to[bend right = 22](3);
\draw[red,very thick](3)to[bend right = 22](5);
\draw[red,very thick](5)to[bend right = 22](19);
\draw[red,very thick](17)to[bend right = 22](15);
\draw[red,very thick](15)to[bend right = 22](13);
\draw[red,very thick](13)to[bend right = 22](11);
\draw[red,very thick](11)to[bend right = 22](9);
\draw[red,very thick](9)to[bend right = 22](7);
\draw[red,very thick](7)to[bend left = 22](11);
\draw[red,very thick](11)to[bend right = 22](1);
\node[fill=none,draw=none,draw = none](a)at(-10:7){};
\node[fill=none,draw=none,draw = none](b)at(-10:6.5){};
\draw[dashed](19) -- (a);
\draw[dashed,red,very thick](19) -- (b);
\node[fill=none,draw=none,draw = none](a')at(265:7){};
\node[fill=none,draw=none,draw = none](b')at(265:6.5){};
\draw[dashed](17) -- (a');
\draw[dashed,red,very thick](17) -- (b');
\draw[very thick, blue] (1) to[bend left = 22] (13);
\draw[very thick, blue] (13) to[bend left =22] (17);
\draw[very thick, blue] (17) -- (1);
\end{tikzpicture}$$

Connect every other vertex in the innermost cycle, beginning at $v_1$. This will form a cycle of length $3r - 1$ where no edges have been used in the triangle decomposition of the graph. Therefore, by the induction hypothesis, additional edges can be added within this cycle to form a graph $G$ where $\eps_\Delta(G) = 1$. \\

$$\begin{tikzpicture}
[scale=.6,auto=right,every node/.style={circle,fill=gray!30, inner sep = 1pt,draw}]
\foreach \labb/\lab /\ang in {1/$v_1$/90,2/$v_2$/75,3/$v_3$/60,4/$v_4$/45,5/$v_5$/30,6/$v_6$/15,
7/$v_{12r+1}$/105}
\node[](\labb)at(\ang:7){\tiny\lab};
\foreach \labb/\lab /\ang in {11/$v_{12r-3}$/165,12/$v_{12r-4}$/180,13/$v_{12r-5}$/195,14/$v_{12r-6}$/210,15/$v_{12r-7}$/225,10/$v_{12r-2}$/150,9/$v_{12r-1}$/135,8/$v_{12r}$/120}
\node[](\labb)at(\ang:7){\tiny\lab};
\foreach \labb/\lab /\ang in {16/$v_{12r-8}$/240,17/$v_{12r-9}$/255}
\node[](\labb)at(\ang:7){\tiny\lab};
\node[](19)at(0:7){\tiny$v_7$};
\draw(17)--(16)--(15)--(14)--(13)--(12)--(11)--(10)--(9)--(8)--(7)--(1)--(2)--(3)--(4)--(5)--(6)--(19);
\draw[red,very thick](1)to[bend right = 22](3);
\draw[red,very thick](3)to[bend right = 22](5);
\draw[red,very thick](5)to[bend right = 22](19);
\draw[red,very thick](17)to[bend right = 22](15);
\draw[red,very thick](15)to[bend right = 22](13);
\draw[red,very thick](13)to[bend right = 22](11);
\draw[red,very thick](11)to[bend right = 22](9);
\draw[red,very thick](9)to[bend right = 22](7);
\draw[red,very thick](7)to[bend left = 22](11);
\draw[red,very thick](11)to[bend right = 22](1);
\node[fill=none,draw=none,draw = none](a)at(-10:7){};
\node[fill=none,draw=none,draw = none](b)at(-10:6.5){};
\draw[dashed](19) -- (a);
\draw[dashed,red,very thick](19) -- (b);
\node[fill=none,draw=none,draw = none](a')at(265:7){};
\node[fill=none,draw=none,draw = none](b')at(265:6.5){};
\draw[dashed](17) -- (a');
\draw[dashed,red,very thick](17) -- (b');
\draw[very thick, blue] (1) to[bend left = 22] (13);
\draw[very thick, blue] (13) to[bend left =22] (17);
\draw[very thick, blue] (17) -- (1);
\node[fill=none,draw=none,draw=none](c')at(275:6){};
\draw[dashed](17) -- (a');
\node[fill=none,draw=none,draw=none](c)at(-10:6){};
\draw[dashed,red,very thick](17) -- (b');
\draw[dashed,cyan,very thick](1) -- (c');
\draw[dashed,cyan,very thick](5) to[bend right = 15](c);
\draw[very thick, cyan] (1)to[bend right = 17](5);
\end{tikzpicture}$$

\emph{Case 2:} Suppose $k = 4r + 1$ then $n = 12r + 4$.\\

Begin by forming a cycle $C_{12r+4}$. Label the vertices going around the cycle $v_1,v_2,...v_{12r+4}$. Form triangles around the inner edge of the cycle by connecting every second vertex with an edge, starting at $v_1$. This will add edges between $v_1$ and $v_3$, $v_3$ and $v_5$,... and, $v_{12r-1}$ and $v_1$. These triangles are included in the triangle decomposition of the graph. Notice that the innermost cycle of this graph has length $6r+2$.

$$\begin{tikzpicture}
[scale=.4,auto=right,every node/.style={circle,fill=gray!30, draw,inner sep = 2pt}]
\foreach \labb/\lab /\ang in {1/$v_1$/90, 2/$v_2$/60, 3/$v_3$/30,4/$v_4$/0, 5/$v_{12r+4}$/120,6/$v_{12r+3}$/150,7/$v_{12r+2}$/180}
{\node(\labb)at(\ang:10){\tiny\lab};}
\node(8)at(210:10){\tiny$v_{12r-3}$};
\node(9)at(-30:10){\tiny$v_5$};
\node[fill=none,draw=none,draw=none](10)at(230:9){};
\node[fill=none,draw=none,draw=none](11)at(230:10){};
\node[fill=none,draw=none,draw=none](12)at(-50:9){};
\node[fill=none,draw=none,draw=none](13)at(-50:10){};
\draw[dashed,red,very thick](10) -- (8);
\draw[dashed](9) -- (13);
\draw[dashed,red,very thick](12) -- (9);
\draw[dashed](8) -- (11);
\draw(1)--(2)--(3)--(4);
\draw(7)--(6)--(5)--(1);
\draw(7) -- (8);
\draw(4)--(9);
\draw[very thick, red] (1) -- (3);
\draw[very thick, red] (9) -- (3);
\draw[very thick, red] (8) -- (6);
\draw[very thick, red] (1) -- (6);
\end{tikzpicture}$$

Again, connect every other vertex on the inner cycle, starting at $v_1$. This forms a cycle of length $3r+1$, where no edges have been used in the triangle decomposition yet. Therefore, by the induction hypothesis, additional edges can be added within this cycle to form a graph $G$ where $\eps_\Delta(G) = 1$.

$$\begin{tikzpicture}
[scale=.4,auto=right,every node/.style={circle,fill=gray!30, draw,inner sep = 2pt}]
\foreach \labb/\lab /\ang in {1/$v_1$/90, 2/$v_2$/60, 3/$v_3$/30,4/$v_4$/0, 5/$v_{12r+4}$/120,6/$v_{12r+3}$/150,7/$v_{12r+2}$/180}
{\node(\labb)at(\ang:10){\tiny\lab};}
\node(8)at(210:10){\tiny$v_{12r-3}$};
\node(9)at(-30:10){\tiny$v_5$};
\node[fill=none,draw=none](10)at(230:9){};
\node[fill=none,draw=none](11)at(230:10){};
\node[fill=none,draw=none](15)at(230:8){};
\node[fill=none,draw=none](12)at(-50:9){};
\node[fill=none,draw=none](13)at(-50:10){};
\node[fill=none,draw=none](14)at(-50:8){};
\draw[dashed,red,very thick](10) -- (8);
\draw[dashed](9) -- (13);
\draw[dashed,red,very thick](12) -- (9);
\draw[dashed](8) -- (11);
\draw(1)--(2)--(3)--(4);
\draw(7)--(6)--(5)--(1);
\draw(7) -- (8);
\draw(4)--(9);
\draw[very thick, red] (1) -- (3);
\draw[very thick, red] (9) -- (3);
\draw[very thick, red] (8) -- (6);
\draw[very thick, red] (1) -- (6);
\draw[very thick, blue] (1) -- (9);
\draw[very thick, blue] (1) -- (8);
\draw[very thick, blue,dashed] (9) -- (14);
\draw[very thick, blue,dashed] (8) -- (15);
\end{tikzpicture}$$

\emph{Case 3:} Suppose $k = 4r+2$ then $n = 12r + 7$. \\

Begin by forming a cycle $C_{12r+7}$. Label the vertices going around the cycle $v_1,v_2,...v_{12r+7}$. Starting at $v_1$, connect every other vertex with an edge, stopping at $v_{12r+7}$. All edges in the graph, other than $v_{1}v_{12r+7}$ now lie on a triangle. To form a triangle with the edge $v_{1}v_{12r+7}$, connect $v_1$ and $v_{12r+7}$ to the vertex $v_{12r+3}$. All edges in the graph now lie on a triangle. These triangles are included in the triangle decomposition of the graph. Notice the innermost cycle of the graph now has order $6r+2$. 

$$\begin{tikzpicture}
[scale=.67,auto=right,every node/.style={circle,fill=gray!30, inner sep= 2pt,draw}]
\foreach \labb/\lab /\ang in {1/$v_1$/90,2/$v_2$/70,3/$v_3$/50,4/$v_4$/30,5/$v_5$/10,6/$v_6$/-10,7/$v_{12r+7}$/110,8/$v_{12r+6}$/130,9/$v_{12r+5}$/150,10/$v_{12r+4}$/170,11/$v_{12r+3}$/190}
\node(\labb)at(\ang:7){\tiny\lab};
\node[fill=none,draw=none](12)at(-20:7){};
\node[fill=none,draw=none](13)at(-20:6.5){};
\node[fill=none,draw=none](14)at(200:7){};
\node[fill=none,draw=none](15)at(200:6.5){};
\draw[dashed](12)--(6);
\draw[dashed,very thick, red](5)to[bend right=15](13);
\draw[dashed](11)--(14);
\draw[dashed,very thick, red](11)--(15);
\draw (11) -- (10) -- (9)--(8)--(7)--(1)--(2)--(3)--(4)--(5)--(6);
\draw[very thick, red] (1) to[bend right=15](3) ;
\draw[very thick, red] (3) to[bend right=15] (5);
\draw[very thick, red] (11) to[bend right=15] (9);
\draw[very thick, red] (9)to[bend right=15] (7);
\draw[very thick, red] (1) to[bend left=15] (11) to[bend right=15] (7);
\end{tikzpicture}$$

Again, connect every other vertex on the inner cycle, starting at $v_1$. This forms a cycle of length $3r+1$, where no edges have been used in the triangle decomposition yet. Therefore, by the induction hypothesis, additional edges can be added within this cycle to form a graph $G$ where $\eps_\Delta(G) = 1$.

$$\begin{tikzpicture}
[scale=.67,auto=right,every node/.style={circle,fill=gray!30, inner sep= 2pt,draw}]
\foreach \labb/\lab /\ang in {1/$v_1$/90,2/$v_2$/70,3/$v_3$/50,4/$v_4$/30,5/$v_5$/10,6/$v_6$/-10,7/$v_{12r+7}$/110,8/$v_{12r+6}$/130,9/$v_{12r+5}$/150,10/$v_{12r+4}$/170,11/$v_{12r+3}$/190}
\node(\labb)at(\ang:7){\tiny\lab};
\node[fill=none,draw = none](12)at(-20:7){};
\node[fill=none,draw = none](13)at(-20:6.5){};
\node[fill=none,draw = none](14)at(200:7){};
\node[fill=none,draw = none](15)at(200:6.5){};
\draw[dashed](12)--(6);
\draw[dashed,very thick, red](5)to[bend right=15](13);
\draw[dashed](11)--(14);
\draw[dashed,very thick, red](11)--(15);
\draw (11) -- (10) -- (9)--(8)--(7)--(1)--(2)--(3)--(4)--(5)--(6);
\draw[very thick, red] (1) to[bend right=15](3) ;
\draw[very thick, red] (3) to[bend right=15] (5);
\draw[very thick, red] (11) to[bend right=15] (9);
\draw[very thick, red] (9)to[bend right=15] (7);
\draw[very thick, red] (1) to[bend left=15] (11) to[bend right=15] (7);
\draw[very thick, blue](1)to[bend right =15](5);
\node[fill=none,draw=none](16)at(200:5.5){};
\draw[very thick, blue,dashed](1)to[bend left =15](16);
\node[fill=none,draw=none](17)at(-20:5.5){};
\draw[dashed,very thick, blue](5)to[bend right =15](17);
\end{tikzpicture}$$ \\~\\

\emph{Case 4:} Suppose $k = 4r+3$ then $n = 12r + 10$. \\

Begin by forming a cycle of order $C_{12r+10}$. Label the vertices going around the cycle $v_1,v_2,...v_{12r+10}$. Starting at $v_1$, connect every other vertex. Every edge in the graph now lies on a triangle. We include all of these triangles in our triangle decomposition. Notice that the innermost cycle of the graph has length $6r+5$. 

$$\begin{tikzpicture}
[scale=.6,auto=right,every node/.style={circle,fill=gray!30, inner sep = 1pt,draw}]
\foreach \labb/\lab /\ang in {1/$v_1$/90,2/$v_2$/70,3/$v_3$/50,4/$v_4$/30,5/$v_5$/10,6/$v_6$/-10,7/$v_{12r+10}$/110,8/$v_{12r+9}$/130,9/$v_{12r+8}$/150,10/$v_{12r+7}$/170,11/$v_{12r+6}$/190,12/$v_{12r+5}$/210,15/$v_{7}$/-30,13/$v_{12r+4}$/230,14/$v_{12r+3}$/250}
\node(\labb)at(\ang:7){\tiny\lab};
\draw(14)--(13)--(12)--(11)--(10)--(9)--(8)--(7)--(1)--(2)--(3)--(4)--(5)--(6)--(15);
\draw[red,very thick](1)to[bend right =15] (3);
\draw[red,very thick](3)to[bend right =15](5);
\draw[red,very thick](5)to[bend right =15](15);
\draw[red,very thick](14)to[bend right =15](12);
\draw[red,very thick](12)to[bend right =15](10);
\draw[red,very thick](10)to[bend right =15](8);
\draw[red,very thick](8)to[bend right =15](1);
\node[fill=none, draw = none](16)at(265:7){};
\draw[dashed](14)--(16);
\node[fill=none, draw = none](17)at(265:6){};
\draw[dashed,red,very thick](14)--(17);
\node[fill=none, draw = none](18)at(-40:7){};
\draw[dashed](15)--(18);
\node[fill=none, draw = none](19)at(-40:6){};
\draw[dashed,red,very thick](15)--(19);
\end{tikzpicture}$$
Next, form a triangle with vertices $v_1$, $v_{12r+7}$, and $v_{12r+5}$. This new triangle is also part of the triangle decomposition. Notice the innermost cycle is now of order $6r + 2$. 

$$\begin{tikzpicture}
[scale=.6,auto=right,every node/.style={circle,fill=gray!30, inner sep = 1pt,draw}]
\foreach \labb/\lab /\ang in {1/$v_1$/90,2/$v_2$/70,3/$v_3$/50,4/$v_4$/30,5/$v_5$/10,6/$v_6$/-10,7/$v_{12r+10}$/110,8/$v_{12r+9}$/130,9/$v_{12r+8}$/150,10/$v_{12r+7}$/170,11/$v_{12r+6}$/190,12/$v_{12r+5}$/210,15/$v_{7}$/-30,13/$v_{12r+4}$/230,14/$v_{12r+3}$/250}
\node(\labb)at(\ang:7){\tiny\lab};
\draw(14)--(13)--(12)--(11)--(10)--(9)--(8)--(7)--(1)--(2)--(3)--(4)--(5)--(6)--(15);
\draw[red,very thick](1)to[bend right =15] (3);
\draw[red,very thick](3)to[bend right =15](5);
\draw[red,very thick](5)to[bend right =15](15);
\draw[red,very thick](14)to[bend right =15](12);
\draw[red,very thick](12)to[bend right =15](10);
\draw[red,very thick](10)to[bend right =15](8);
\draw[red,very thick](8)to[bend right =15](1);
\node[fill=none, draw = none](16)at(265:7){};
\draw[dashed](14)--(16);
\node[fill=none, draw = none](17)at(265:6){};
\draw[dashed,red,very thick](14)--(17);
\node[fill=none, draw = none](18)at(-40:7){};
\draw[dashed](15)--(18);
\node[fill=none, draw = none](19)at(-40:6){};
\draw[dashed,red,very thick](15)--(19);
\node[fill=none, draw = none](18)at(-40:7){};
\draw[dashed](15)--(18);
\node[fill=none, draw = none](19)at(-40:6){};
\draw[very thick, blue](14)to[bend left =0](1);
\draw[very thick, blue](1)to[bend left =5](10);
\draw[very thick, blue](10)to[bend left =5](14);
\end{tikzpicture}$$
Again, connect every other vertex of the inner cycle starting at $v_1$. This forms a cycle of length $3r + 1$ where no edges have been used in the triangle decomposition yet. As $r < k$, the result follows from the induction hypothesis. 
$$\begin{tikzpicture}
[scale=.6,auto=right,every node/.style={circle,fill=gray!30, inner sep = 1pt,draw}]
\foreach \labb/\lab /\ang in {1/$v_1$/90,2/$v_2$/70,3/$v_3$/50,4/$v_4$/30,5/$v_5$/10,6/$v_6$/-10,7/$v_{12r+10}$/110,8/$v_{12r+9}$/130,9/$v_{12r+8}$/150,10/$v_{12r+7}$/170,11/$v_{12r+6}$/190,12/$v_{12r+5}$/210,15/$v_{7}$/-30,13/$v_{12r+4}$/230,14/$v_{12r+3}$/250}
\node(\labb)at(\ang:7){\tiny\lab};
\draw(14)--(13)--(12)--(11)--(10)--(9)--(8)--(7)--(1)--(2)--(3)--(4)--(5)--(6)--(15);
\draw[red,very thick](1)to[bend right =15] (3);
\draw[red,very thick](3)to[bend right =15](5);
\draw[red,very thick](5)to[bend right =15](15);
\draw[red,very thick](14)to[bend right =15](12);
\draw[red,very thick](12)to[bend right =15](10);
\draw[red,very thick](10)to[bend right =15](8);
\draw[red,very thick](8)to[bend right =15](1);
\node[fill=none, draw = none](16)at(265:7){};
\draw[dashed](14)--(16);
\node[fill=none, draw = none](17)at(265:6){};
\draw[dashed,red,very thick](14)--(17);
\node[fill=none, draw = none](18)at(-40:7){};
\draw[dashed](15)--(18);
\node[fill=none, draw = none](19)at(-40:6){};
\draw[dashed,red,very thick](15)--(19);
\node[fill=none, draw = none](18)at(-40:7){};
\draw[dashed](15)--(18);
\node[fill=none, draw = none](19)at(-40:6){};
\draw[very thick, blue](14)to[bend left =0](1);
\draw[very thick, blue](1)to[bend left =5](10);
\draw[very thick, blue](10)to[bend left =5](14);
\draw[very thick, cyan] (1) to[bend right =10](5);
\draw[very thick, cyan, dashed](5)to[bend right =10](19);
\draw[very thick, cyan, dashed](1)to[bend left =10](17);
\end{tikzpicture}$$

\end{proof}

\textbf{Theorem 7.} \emph{ For $n \geq 3$, $n \equiv 2 \mod 3$ if and only if $\eps_{\Delta}(\mathcal{OP}_{n}) = 2$.}
\begin{proof}

Suppose first that $\eps_{\Delta}(\mathcal{OP}_{n}) = 2$. Then, there exists $G \in \mathcal{OP}_n$ such that $\eps_{\Delta}(G) = 2$. Let us fix such a $G$. Then $E(G) + 2 \equiv 0 \text{ (mod 3)}$. Therefore, $2n -3 + 2 \equiv 0 \text{ (mod 3)}$. This implies $n \equiv 2 \text{ (mod 3)}$, as desired.  \\

Conversely, suppose $n \equiv 2 \text{ (mod 3)}$. We will show there exists some graph $G \in \mathcal{OP}_n$ where $\eps_\Delta(G) = 2$. It is convenient to look at four cases, one for each congruence class mod 4. We proceed by induction on $n$. \\

First, consider the following graphs on $n$ vertices, where $n = 5, 8, 11, 14$. The order these graphs represents the congruence classes mod 4. We include an example of a graph $G$ of order 17 where $\eps_\Delta(G) = 2$, as the construction in our induction step does not work for order 17. \\

$$n = 5: \\
\begin{tikzpicture}
[scale=.75,auto=right,every node/.style={circle,fill=gray!30, draw}]
\foreach \lab/\ang in {a/90,b/162,c/234,d/306,e/378}
{\node(\lab) at(\ang:2){};}
\draw[very thick, red](a) -- (b) -- (c) -- (a);
\draw[very thick, blue](c) -- (d) -- (e) -- (c);
\draw[very thick, green](a) -- (e);
\draw[very thick, green](a) to[out = 280, in = 60] (c); 
\draw[very thick, green](e) to[in = 45, out = 190] (c); 
\end{tikzpicture}\hspace{.5cm}
n = 8: \\
\begin{tikzpicture}
[scale=1,auto=right,every node/.style={circle,fill=gray!30, draw}]
\foreach \lab/\ang in {a/0,b/45,c/90,d/135,e/180,f/225,g/270,h/315}
{\node(\lab) at(\ang:2){};}
\draw[very thick, red](a) -- (b) --(c)--(a);
\draw[very thick, green](e) -- (d) --(c)--(e);
\draw[very thick, blue](e) -- (f) --(g)--(e);
\draw[very thick, orange](a) -- (h) --(g)--(a);
\draw[very thick, pink](a) -- (e);
\draw[very thick, pink](a) to[out = 155, in = 285] (c);
\draw[very thick, pink](c) to[out = 255, in = 25] (e);
\end{tikzpicture}$$

$$n = 11:\\
\begin{tikzpicture}
[scale=1,auto=right,every node/.style={circle,fill=gray!30, draw}]
\foreach \lab/\ang in {a/0,b/32.72,c/65.44,d/98.17,e/130.89,f/163.61,g/196.33,h/229.05,i/261.77,j/294.49,k/327.21}
{\node(\lab) at(\ang:3){};}
\draw[very thick, red](a) -- (b) --(c)--(a);
\draw[very thick, green](e) -- (d) --(c)--(e);
\draw[very thick, blue](e) -- (f) --(g)--(e);
\draw[very thick, orange](g) -- (h) -- (i) -- (g);
\draw[very thick, pink](k) -- (j) -- (i) -- (k);
\draw[very thick, cyan](a) -- (k) -- (e) -- (a);
\draw[very thick, brown](k) -- (g);
\draw[very thick, brown](k) to[out=160,in=300] (e);
\draw[very thick, brown](g) to[in=280, out=45] (e);
\end{tikzpicture} \hspace{.5cm}
n = 14:\\
\begin{tikzpicture}
[scale=1,auto=right,every node/.style={circle,fill=gray!30, draw}]
\foreach \lab/\ang in {a/0,b/25.714,c/51.43,d/77.1,e/102.86,f/128.57,g/154.28,h/180,i/205.71,j/231.42,k/257.14,l/282.85,m/308.57,n/334.28}
{\node(\lab) at(\ang:4){};}
\draw[very thick, red](a) -- (b) --(c)--(a);
\draw[very thick, green](e) -- (d) --(c)--(e);
\draw[very thick, blue](e) -- (f) --(g)--(e);
\draw[very thick, orange](g) -- (h) -- (i) -- (g);
\draw[very thick, pink](k) -- (j) -- (i) -- (k);
\draw[very thick, brown](k) -- (l) -- (m) -- (k);
\draw[very thick, cyan](m) -- (n) -- (a) -- (m);
\draw[very thick, purple](a) -- (e) -- (i) -- (a);
\draw[very thick, black](a) -- (k);
\draw[very thick, black] (k) to[out= 130, in= 340] (i);
\draw[very thick, black] (i) to[out= 0, in= 205] (a);
\end{tikzpicture}$$

$$n = 17: \\\begin{tikzpicture}
[scale=.7,auto=right,every node/.style={circle,fill=gray!30, draw}]
\foreach \lab/\ang in {a/0,b/21.18,c/42.36,d/63.54,e/84.72,f/105.9,g/127.08,h/148.26,i/169.44,j/190.62,k/211.8,l/232.98,m/254.156,n/275.53,o/296.52,p/317.7,q/338.876}
{\node(\lab) at(\ang:4.5){};}
\draw[very thick, red](a)--(b)--(c)--(a);
\draw[very thick, blue](e)--(d)--(c)--(e);
\draw[very thick, green](e)--(f)--(g)--(e);
\draw[very thick, black](g)--(h)--(i)--(g);
\draw[very thick, pink](i)--(j)--(k)--(i);
\draw[very thick, brown](k)--(l)--(m)--(k);
\draw[very thick, yellow](m)--(n)--(o)--(m);
\draw[very thick, orange](o)--(p)--(q)--(o);
\draw[very thick, cyan](q)--(a)--(e)--(q);
\draw[very thick, purple](q)--(g)--(k)--(q);
\draw[very thick, violet](q)--(m);
\draw[very thick, violet](m)to[in=340,out=120](k);
\draw[very thick, violet](q)to[in=355,out=200](k);
\end{tikzpicture}$$

Next, assume we there exists a maximal outerplanar graph $G$ of order $3(k-1) + 2$ where $\eps_\Delta(G) = 2$, for some $k \geq 2$. Recall that all maximal outerplanar graphs are Hamiltonian, so this assumption is equivalent to saying the inside of any cycle of length $3(k-1) + 2$ can be triangulated to form a graph $G$ where $\eps_\Delta(G) = 2$. \\

\emph{Case 1:} Suppose $k = 4r$ then $n = 12r + 2$ \\

Begin by forming a cycle of order $C_{12r+2}$. Label the vertices going around the cycle $v_1,v_2,...v_{12r+2}$. Starting at $v_1$, connect every other vertex. Every edge in the graph now lies on a triangle. We include all of these triangles in our triangle decomposition. Notice that the innermost cycle of the graph has length $6r+1$. 

$$\begin{tikzpicture}
[scale=.6,auto=right,every node/.style={circle,fill=gray!30, inner sep = 1pt,draw}]
\foreach \labb/\lab /\ang in {1/$v_1$/90,2/$v_2$/70,3/$v_3$/50,4/$v_4$/30,5/$v_5$/10,6/$v_6$/-10,7/$v_{12r+2}$/110,8/$v_{12r+1}$/130,9/$v_{12r}$/150,10/$v_{12r-1}$/170,11/$v_{12r-2}$/190,12/$v_{12r-3}$/210,15/$v_{7}$/-30,13/$v_{12r-4}$/230,14/$v_{12r-5}$/250}
\node(\labb)at(\ang:7){\tiny\lab};
\draw(14)--(13)--(12)--(11)--(10)--(9)--(8)--(7)--(1)--(2)--(3)--(4)--(5)--(6)--(15);
\draw[red,very thick](1)to[bend right =15] (3);
\draw[red,very thick](3)to[bend right =15](5);
\draw[red,very thick](5)to[bend right =15](15);
\draw[red,very thick](14)to[bend right =15](12);
\draw[red,very thick](12)to[bend right =15](10);
\draw[red,very thick](10)to[bend right =15](8);
\draw[red,very thick](8)to[bend right =15](1);
\node[fill=none, draw = none](16)at(265:7){};
\draw[dashed](14)--(16);
\node[fill=none, draw = none](17)at(265:6){};
\draw[dashed,red,very thick](14)--(17);
\node[fill=none, draw = none](18)at(-40:7){};
\draw[dashed](15)--(18);
\node[fill=none, draw = none](19)at(-40:6){};
\draw[dashed,red,very thick](15)--(19);
\end{tikzpicture}$$

Next, form a triangle with vertices $v_1$, $v_{12r-1}$, and $v_{12r-3}$. This new triangle is also part of the triangle decomposition. Notice the innermost cycle is now of order $6r - 2$. 

$$\begin{tikzpicture}
[scale=.6,auto=right,every node/.style={circle,fill=gray!30, inner sep = 1pt,draw}]
\foreach \labb/\lab /\ang in {1/$v_1$/90,2/$v_2$/70,3/$v_3$/50,4/$v_4$/30,5/$v_5$/10,6/$v_6$/-10,7/$v_{12r+2}$/110,8/$v_{12r+1}$/130,9/$v_{12r}$/150,10/$v_{12r-1}$/170,11/$v_{12r-2}$/190,12/$v_{12r-3}$/210,15/$v_{7}$/-30,13/$v_{12r-4}$/230,14/$v_{12r-5}$/250}
\node(\labb)at(\ang:7){\tiny\lab};
\draw(14)--(13)--(12)--(11)--(10)--(9)--(8)--(7)--(1)--(2)--(3)--(4)--(5)--(6)--(15);
\draw[red,very thick](1)to[bend right =15] (3);
\draw[red,very thick](3)to[bend right =15](5);
\draw[red,very thick](5)to[bend right =15](15);
\draw[red,very thick](14)to[bend right =15](12);
\draw[red,very thick](12)to[bend right =15](10);
\draw[red,very thick](10)to[bend right =15](8);
\draw[red,very thick](8)to[bend right =15](1);
\node[fill=none, draw = none](16)at(265:7){};
\draw[dashed](14)--(16);
\node[fill=none, draw = none](17)at(265:6){};
\draw[dashed,red,very thick](14)--(17);
\node[fill=none, draw = none](18)at(-40:7){};
\draw[dashed](15)--(18);
\node[fill=none, draw = none](19)at(-40:6){};
\draw[dashed,red,very thick](15)--(19);
\node[fill=none, draw = none](18)at(-40:7){};
\draw[dashed](15)--(18);
\node[fill=none, draw = none](19)at(-40:6){};
\draw[very thick, blue](14)to[bend left =0](1);
\draw[very thick, blue](1)to[bend left =5](10);
\draw[very thick, blue](10)to[bend left =5](14);
\end{tikzpicture}$$
 
 Again, connect every other vertex of the inner cycle starting at $v_1$. This forms a cycle of length $3r - 1$ where no edges have been used in the triangle decomposition yet. As $r < k$, the result follows from the induction hypothesis. 
 
 $$\begin{tikzpicture}
[scale=.6,auto=right,every node/.style={circle,fill=gray!30, inner sep = 1pt,draw}]
\foreach \labb/\lab /\ang in {1/$v_1$/90,2/$v_2$/70,3/$v_3$/50,4/$v_4$/30,5/$v_5$/10,6/$v_6$/-10,7/$v_{12r+2}$/110,8/$v_{12r+1}$/130,9/$v_{12r}$/150,10/$v_{12r-1}$/170,11/$v_{12r-2}$/190,12/$v_{12r-3}$/210,15/$v_{7}$/-30,13/$v_{12r-4}$/230,14/$v_{12r-5}$/250}
\node(\labb)at(\ang:7){\tiny\lab};
\draw(14)--(13)--(12)--(11)--(10)--(9)--(8)--(7)--(1)--(2)--(3)--(4)--(5)--(6)--(15);
\draw[red,very thick](1)to[bend right =15] (3);
\draw[red,very thick](3)to[bend right =15](5);
\draw[red,very thick](5)to[bend right =15](15);
\draw[red,very thick](14)to[bend right =15](12);
\draw[red,very thick](12)to[bend right =15](10);
\draw[red,very thick](10)to[bend right =15](8);
\draw[red,very thick](8)to[bend right =15](1);
\node[fill=none, draw = none](16)at(265:7){};
\draw[dashed](14)--(16);
\node[fill=none, draw = none](17)at(265:6){};
\draw[dashed,red,very thick](14)--(17);
\node[fill=none, draw = none](18)at(-40:7){};
\draw[dashed](15)--(18);
\node[fill=none, draw = none](19)at(-40:6){};
\draw[dashed,red,very thick](15)--(19);
\node[fill=none, draw = none](18)at(-40:7){};
\draw[dashed](15)--(18);
\node[fill=none, draw = none](19)at(-40:6){};
\draw[very thick, blue](14)to[bend left =0](1);
\draw[very thick, blue](1)to[bend left =5](10);
\draw[very thick, blue](10)to[bend left =5](14);
\draw[very thick, cyan] (1) to[bend right =10](5);
\draw[very thick, cyan, dashed](5)to[bend right =10](19);
\draw[very thick, cyan, dashed](1)to[bend left =10](17);
\end{tikzpicture}$$

\emph{Case 2:} Suppose $k = 4r+1$ then $n = 12r+ 5$ (from our base case $r \geq 2$) \\

Begin by forming a cycle $C_{12r+5}$. Connect every other vertex in the graph beginning at $v_1$ until $v_{12r+5}$. Every edge other than $v_1v_{12r+5}$ lies in a triangle. Form a triangle between $v_{1}$, $v_{12r+5}$ and $v_{12r+1}$ by adding an edge between $v_1$ and $v_{12r-3}$ and an edge between $v_{12r+5}$ and $v_{12r+1}$. The inner-cycle of the graph has order $6r+1$. 

$$\begin{tikzpicture}
[scale=.7,auto=right,every node/.style={circle,fill=gray!30, inner sep = 1pt,draw}]
\foreach \labb/\lab /\ang in {1/$v_1$/90,2/$v_2$/75,3/$v_3$/60,4/$v_4$/45,5/$v_5$/30,6/$v_6$/15,
7/$v_{12r+5}$/105,8/$v_{12r+4}$/120,9/$v_{12r+3}$/135,10/$v_{12r+2}$/150}
\node(\labb)at(\ang:7){\tiny\lab};
\foreach \labb/\lab /\ang in {11/$v_{12r-1}$/165,12/$v_{12r-2}$/180,13/$v_{12r-3}$/195,14/$v_{12r-4}$/210,15/$v_{12r-5}$/225}
\node(\labb)at(\ang:7){\tiny\lab};
\foreach \labb/\lab /\ang in {16/$v_{12r-6}$/240,17/$v_{12r-7}$/255}
\node(\labb)at(\ang:7){\tiny\lab};
\node(19)at(0:7){\tiny$v_7$};
\draw(17)--(16)--(15)--(14)--(13)--(12)--(11)--(10)--(9)--(8)--(7)--(1)--(2)--(3)--(4)--(5)--(6)--(19);
\draw[red,very thick](1)to[bend right = 22](3);
\draw[red,very thick](3)to[bend right = 22](5);
\draw[red,very thick](5)to[bend right = 22](19);
\draw[red,very thick](17)to[bend right = 22](15);
\draw[red,very thick](15)to[bend right = 22](13);
\draw[red,very thick](13)to[bend right = 22](11);
\draw[red,very thick](11)to[bend right = 22](9);
\draw[red,very thick](9)to[bend right = 22](7);
\draw[red,very thick](7)to[bend left = 22](11);
\draw[red,very thick](11)to[bend right = 22](1);
\node[fill=none,draw=none,draw = none](a)at(-10:7){};
\node[fill=none,draw=none,draw = none](b)at(-10:6.5){};
\draw[dashed](19) -- (a);
\draw[dashed,red,very thick](19) -- (b);
\node[fill=none,draw=none,draw = none](a')at(265:7){};
\node[fill=none,draw=none,draw = none](b')at(265:6.5){};
\draw[dashed](17) -- (a');
\draw[dashed,red,very thick](17) -- (b');
\end{tikzpicture}$$

Form a triangle with vertices $v_1$, $v_{12r-3}$ and $v_{12r-7}$. This new triangle is included in our triangle decomposition. The innermost cycle is of order $6r - 2$. \\

$$\begin{tikzpicture}
[scale=.7,auto=right,every node/.style={circle,fill=gray!30, inner sep = 1pt,draw}]
\foreach \labb/\lab /\ang in {1/$v_1$/90,2/$v_2$/75,3/$v_3$/60,4/$v_4$/45,5/$v_5$/30,6/$v_6$/15,
7/$v_{12r+5}$/105,8/$v_{12r+4}$/120,9/$v_{12r+3}$/135,10/$v_{12r+2}$/150}
\node(\labb)at(\ang:7){\tiny\lab};
\foreach \labb/\lab /\ang in {11/$v_{12r-1}$/165,12/$v_{12r-2}$/180,13/$v_{12r-3}$/195,14/$v_{12r-4}$/210,15/$v_{12r-5}$/225}
\node(\labb)at(\ang:7){\tiny\lab};
\foreach \labb/\lab /\ang in {16/$v_{12r-6}$/240,17/$v_{12r-7}$/255}
\node(\labb)at(\ang:7){\tiny\lab};
\node(19)at(0:7){\tiny$v_7$};
\draw(17)--(16)--(15)--(14)--(13)--(12)--(11)--(10)--(9)--(8)--(7)--(1)--(2)--(3)--(4)--(5)--(6)--(19);
\draw[red,very thick](1)to[bend right = 22](3);
\draw[red,very thick](3)to[bend right = 22](5);
\draw[red,very thick](5)to[bend right = 22](19);
\draw[red,very thick](17)to[bend right = 22](15);
\draw[red,very thick](15)to[bend right = 22](13);
\draw[red,very thick](13)to[bend right = 22](11);
\draw[red,very thick](11)to[bend right = 22](9);
\draw[red,very thick](9)to[bend right = 22](7);
\draw[red,very thick](7)to[bend left = 22](11);
\draw[red,very thick](11)to[bend right = 22](1);
\node[fill=none,draw = none](a)at(-10:7){};
\node[fill=none,draw = none](b)at(-10:6.5){};
\draw[dashed](19) -- (a);
\draw[dashed,red,very thick](19) -- (b);
\node[fill=none,draw = none](a')at(265:7){};
\node[fill=none,draw = none](b')at(265:6.5){};
\draw[dashed](17) -- (a');
\draw[dashed,red,very thick](17) -- (b');
\draw[very thick, blue] (1) to[bend left = 22] (13);
\draw[very thick, blue] (13) to[bend left =22] (17);
\draw[very thick, blue] (17) -- (1);
\end{tikzpicture}$$

Connect every other vertex in the innermost cycle, beginning at $v_1$. This will form a cycle of length $3r - 1$ where no edges have been used in the triangle decomposition of the graph. As $r < k$, the result follows from the induction hypothesis. 

$$\begin{tikzpicture}
[scale=.7,auto=right,every node/.style={circle,fill=gray!30, inner sep = 1pt,draw}]
\foreach \labb/\lab /\ang in {1/$v_1$/90,2/$v_2$/75,3/$v_3$/60,4/$v_4$/45,5/$v_5$/30,6/$v_6$/15,
7/$v_{12r+5}$/105,8/$v_{12r+4}$/120,9/$v_{12r+3}$/135,10/$v_{12r+2}$/150}
\node(\labb)at(\ang:7){\tiny\lab};
\foreach \labb/\lab /\ang in {11/$v_{12r-1}$/165,12/$v_{12r-2}$/180,13/$v_{12r-3}$/195,14/$v_{12r-4}$/210,15/$v_{12r-5}$/225}
\node(\labb)at(\ang:7){\tiny\lab};
\foreach \labb/\lab /\ang in {16/$v_{12r-6}$/240,17/$v_{12r-7}$/255}
\node(\labb)at(\ang:7){\tiny\lab};
\node(19)at(0:7){\tiny$v_7$};
\draw(17)--(16)--(15)--(14)--(13)--(12)--(11)--(10)--(9)--(8)--(7)--(1)--(2)--(3)--(4)--(5)--(6)--(19);
\draw[red,very thick](1)to[bend right = 22](3);
\draw[red,very thick](3)to[bend right = 22](5);
\draw[red,very thick](5)to[bend right = 22](19);
\draw[red,very thick](17)to[bend right = 22](15);
\draw[red,very thick](15)to[bend right = 22](13);
\draw[red,very thick](13)to[bend right = 22](11);
\draw[red,very thick](11)to[bend right = 22](9);
\draw[red,very thick](9)to[bend right = 22](7);
\draw[red,very thick](7)to[bend left = 22](11);
\draw[red,very thick](11)to[bend right = 22](1);
\node[fill=none,draw = none](a)at(-10:7){};
\node[fill=none,draw = none](b)at(-10:6.5){};
\draw[dashed](19) -- (a);
\draw[dashed,red,very thick](19) -- (b);
\node[fill=none,draw = none](a')at(265:7){};
\node[fill=none,draw = none](b')at(265:6.5){};
\draw[dashed](17) -- (a');
\draw[dashed,red,very thick](17) -- (b');
\draw[very thick, blue] (1) to[bend left = 22] (13);
\draw[very thick, blue] (13) to[bend left =22] (17);
\draw[very thick, blue] (17) -- (1);
\node[fill=none,draw=none](c')at(275:6){};
\draw[dashed](17) -- (a');
\node[fill=none,draw=none](c)at(-10:6){};
\draw[dashed,red,very thick](17) -- (b');
\draw[dashed,cyan,very thick](1) -- (c');
\draw[dashed,cyan,very thick](5) to[bend right = 15](c);
\draw[very thick, cyan] (1)to[bend right = 17](5);
\end{tikzpicture}$$

\emph{Case 3:} Suppose $k = 4r + 2$ then $n = 12r + 8$. \\

Begin by forming a cycle $C_{12r + 8}$. Label the vertices going around the cycle $v_1,v_2,...v_{12r + 8}$. Form triangles around the inner edge of the cycle by connecting every second vertex with an edge, starting at $v_1$. This will add edges between $v_1$ and $v_3$, $v_3$ and $v_5$,... and, $v_{12r-1}$ and $v_1$. These triangles are included in the triangle decomposition of the graph. Notice that the innermost cycle of this graph has length $6r + 4$. 

$$\begin{tikzpicture}
[scale=.4,auto=right,every node/.style={circle,fill=gray!30, draw,inner sep = 2pt}]
\foreach \labb/\lab /\ang in {1/$v_1$/90, 2/$v_2$/60, 3/$v_3$/30,4/$v_4$/0, 5/$v_{12r+8}$/120,6/$v_{12r+7}$/150,7/$v_{12r+6}$/180}
{\node(\labb)at(\ang:10){\tiny\lab};}
\node(8)at(210:10){\tiny$v_{12r-3}$};
\node(9)at(-30:10){\tiny$v_5$};
\node[fill=none,draw=none](10)at(230:9){};
\node[fill=none,draw=none](11)at(230:10){};
\node[fill=none,draw=none](12)at(-50:9){};
\node[fill=none,draw=none](13)at(-50:10){};
\draw[dashed,red,very thick](10) -- (8);
\draw[dashed](9) -- (13);
\draw[dashed,red,very thick](12) -- (9);
\draw[dashed](8) -- (11);
\draw(1)--(2)--(3)--(4);
\draw(7)--(6)--(5)--(1);
\draw(7) -- (8);
\draw(4)--(9);
\draw[very thick, red] (1) -- (3);
\draw[very thick, red] (9) -- (3);
\draw[very thick, red] (8) -- (6);
\draw[very thick, red] (1) -- (6);
\end{tikzpicture}$$

Again, connect every other vertex on the inner cycle, starting at $v_1$. This forms a cycle of length $3r + 2$, where no edges have been used in the triangle decomposition yet. Therefore, the result follows by the induction hypothesis.  

$$\begin{tikzpicture}
[scale=.4,auto=right,every node/.style={circle,fill=gray!30, draw,inner sep = 2pt}]
\foreach \labb/\lab /\ang in {1/$v_1$/90, 2/$v_2$/60, 3/$v_3$/30,4/$v_4$/0, 5/$v_{12r+8}$/120,6/$v_{12r+7}$/150,7/$v_{12r+6}$/180}
{\node(\labb)at(\ang:10){\tiny\lab};}
\node(8)at(210:10){\tiny$v_{12r-3}$};
\node(9)at(-30:10){\tiny$v_5$};
\node[fill=none,draw=none](10)at(230:9){};
\node[fill=none,draw=none](11)at(230:10){};
\node[fill=none,draw=none](15)at(230:8){};
\node[fill=none,draw=none](12)at(-50:9){};
\node[fill=none,draw=none](13)at(-50:10){};
\node[fill=none,draw=none](14)at(-50:8){};
\draw[dashed,red,very thick](10) -- (8);
\draw[dashed](9) -- (13);
\draw[dashed,red,very thick](12) -- (9);
\draw[dashed](8) -- (11);
\draw(1)--(2)--(3)--(4);
\draw(7)--(6)--(5)--(1);
\draw(7) -- (8);
\draw(4)--(9);
\draw[very thick, red] (1) -- (3);
\draw[very thick, red] (9) -- (3);
\draw[very thick, red] (8) -- (6);
\draw[very thick, red] (1) -- (6);
\draw[very thick, blue] (1) -- (9);
\draw[very thick, blue] (1) -- (8);
\draw[very thick, blue,dashed] (9) -- (14);
\draw[very thick, blue,dashed] (8) -- (15);
\end{tikzpicture}$$

\emph{Case 4:} Suppose $k = 4r + 3$ then $n = 12r + 11$. \\

Begin by forming a cycle $C_{12r+11}$. Label the vertices going around the cycle $v_1,v_2,...v_{12r+11}$. Starting at $v_1$, connect every other vertex with an edge, stopping at $v_{12r+11}$. All edges in the graph, other than $v_{1}v_{12r+11}$ now lie on a triangle. To form a triangle with the edge $v_{1}v_{12r+11}$, connect $v_1$ and $v_{12r+11}$ to the vertex $v_{12r_7}$. All edges in the graph now lie on a triangle. These triangles are included in the triangle decomposition of the graph. Notice the innermost cycle of the graph now has order $6r + 4$. 

$$\begin{tikzpicture}
[scale=.7,auto=right,every node/.style={circle,fill=gray!30, inner sep= 2pt,draw}]
\foreach \labb/\lab /\ang in {1/$v_1$/90,2/$v_2$/70,3/$v_3$/50,4/$v_4$/30,5/$v_5$/10,6/$v_6$/-10,7/$v_{12r+11}$/110,8/$v_{12r+10}$/130,9/$v_{12r+9}$/150,10/$v_{12r+8}$/170,11/$v_{12r+7}$/190}
\node(\labb)at(\ang:7){\tiny\lab};
\node[fill=none,draw = none](12)at(-20:7){};
\node[fill=none,draw = none](13)at(-20:6.5){};
\node[fill=none,draw = none](14)at(200:7){};
\node[fill=none,draw = none](15)at(200:6.5){};
\draw[dashed](12)--(6);
\draw[dashed,very thick, red](5)to[bend right=15](13);
\draw[dashed](11)--(14);
\draw[dashed,very thick, red](11)--(15);
\draw (11) -- (10) -- (9)--(8)--(7)--(1)--(2)--(3)--(4)--(5)--(6);
\draw[very thick, red] (1) to[bend right=15](3) ;
\draw[very thick, red] (3) to[bend right=15] (5);
\draw[very thick, red] (11) to[bend right=15] (9);
\draw[very thick, red] (9)to[bend right=15] (7);
\draw[very thick, red] (1) to[bend left=15] (11) to[bend right=15] (7);
\end{tikzpicture}$$

Again, connect every other vertex on the inner cycle, starting at $v_1$. This forms a cycle of length $3r + 2$, where no edges have been used in the triangle decomposition yet. Therefore, the result follows by the induction hypothesis. \\

$$\begin{tikzpicture}
[scale=.7,auto=right,every node/.style={circle,fill=gray!30, inner sep= 2pt,draw}]
\foreach \labb/\lab /\ang in {1/$v_1$/90,2/$v_2$/70,3/$v_3$/50,4/$v_4$/30,5/$v_5$/10,6/$v_6$/-10,7/$v_{12r+11}$/110,8/$v_{12r+10}$/130,9/$v_{12r+9}$/150,10/$v_{12r+8}$/170,11/$v_{12r+7}$/190}
\node(\labb)at(\ang:7){\tiny\lab};
\node[fill=none,draw = none](12)at(-20:7){};
\node[fill=none,draw = none](13)at(-20:6.5){};
\node[fill=none,draw = none](14)at(200:7){};
\node[fill=none,draw = none](15)at(200:6.5){};
\draw[dashed](12)--(6);
\draw[dashed,very thick, red](5)to[bend right=15](13);
\draw[dashed](11)--(14);
\draw[dashed,very thick, red](11)--(15);
\draw (11) -- (10) -- (9)--(8)--(7)--(1)--(2)--(3)--(4)--(5)--(6);
\draw[very thick, red] (1) to[bend right=15](3) ;
\draw[very thick, red] (3) to[bend right=15] (5);
\draw[very thick, red] (11) to[bend right=15] (9);
\draw[very thick, red] (9)to[bend right=15] (7);
\draw[very thick, red] (1) to[bend left=15] (11) to[bend right=15] (7);
\draw[very thick, blue](1)to[bend right =15](5);
\node[fill=none,draw=none](16)at(200:5.5){};
\draw[very thick, blue,dashed](1)to[bend left =15](16);
\node[fill=none,draw=none](17)at(-20:5.5){};
\draw[dashed,very thick, blue](5)to[bend right =15](17);
\end{tikzpicture}$$ \\~\\

\end{proof}

\textbf{Theorem 8.} \emph{Suppose $\Xi_{\Delta}(\mathcal{OP}_{n})=\varepsilon_{\Delta}(\mathcal{OP}%
_{n})+3k$ for some $k \in \Z$. Then, there exists a graph $G_r$ such that $\eps_\Delta(G_r) = \eps_\Delta(\mathcal{OP}) + 3r$ where $r = 1, 2, ... k - 1$.} 

\begin{proof}

Let us work with each conjugacy class for $n \mod 3$. \\

\textbf{Case 1:} $n \equiv 0 \mod 3$. \\

We form a graph $G_r$ where $\eps_\Delta(G_r) = n - 3r$ for some $r$. We begin by forming a $C_n$. Then, label vertices in the graph $v_1, v_2,...v_n$ going around the cycle. We connect $v_1$ to $v_2,v_3,...v_{3r+2}$. This induced subgraph $H = v_1,v_2,...v_{3r+2}$ requires $3r+2-3 = 3r - 1$ edges to form a triangle decomposition, from the proof of Theorem 7.  Next, we add one multi-edge between $v_1$ and $v_{3r+2}$. The induced subgraph of of $v_1, v_{3r+2}, v_{3r+3},..v_n$ ignoring the edge between $v_1$ and $v_{3r+2}$ included in the triangle decomposition of $H$ has order divisible by 3. Thus, a triangle decomposition can be formed without any multi-edges required by the proof of Theorem 4. Thus, $\eps_\Delta(G_r) = 3r - 1 + 1 = 3r$, as desired. \\

\textbf{Case 2:} $n \equiv 2 \mod 3$. \\

We form a graph $G_r$ where $\eps_\Delta(G_r) = 3r + 2$ for some $r$. We begin by forming a $C_n$. Label vertices in the graph $v_1, v_2,...v_n$ going around the cycle. We connect $v_1$ to $v_2,v_3,...v_{3r+2}$. This induced subgraph $H = v_1,v_2,...v_{3r+2}$ will requires $3r+2-3 = 3r - 1$ edges to form a triangle decomposition.  Next, if we add one multi-edge between $v_1$ and $v_{3r+2}$ then the induced subgraph of $v_1, v_{3r+2}, v_{3r+3},..v_n$, ignoring the edge between $v_1$ and $v_{3r+2}$ included in the triangle decomposition of $H$, will have order congruent to $1 \mod 3$. Thus, a triangle decomposition can be formed with the addition of 2 multi-edges, by Corollary 6. Thus, $\eps_\Delta(G_r) = 3r - 1 + 1 +2 = 3r + 2$, as desired. \\

\textbf{Case 3:} $n \equiv 1 \mod 3$. \\

We form a graph $G_r$ where $\eps_\Delta(G_r) = 3r + 1$ for some $r$. Form a $C_n$. Then, label vertices in the graph $v_1, v_2,...v_n$ going around the cycle. We connect $v_1$ to $v_2,v_3,...v_{3r+2}$. This induced subgraph $H = v_1,v_2,...v_{3r+2}$ will require $3r+2-3 = 3r - 1$ edges to form a triangle decomposition, by the proof of Theorem 7.  Next, we add one multi-edge between $v_1$ and $v_{3r+2}$. The induced subgraph of of $v_1, v_{3r+2}, v_{3r+3},..v_n$ ignoring the edge between $v_1$ and $v_{3r+2}$ included in the triangle decomposition of $H$, has order congruent to $2 \mod 3$. Thus, applying Corollary 7 we can add edges to this cycle to form a graph with $\eps_\Delta$ value 2. Thus, $\eps_\Delta(G_r) = 3r - 1 + 1 +1 = 3r + 1$, as desired. 
\end{proof}


\newpage

\begin{thebibliography}{9}                                                                                                
\bibitem {kieka}C.~M.~Mynhardt and C.~M.~van Bommel, Triangle decompositions of
planar graphs, \emph{Discuss. Math. Graph Theory} \textbf{36} (2016), 643-659.

\bibitem{postle}
M. Delcourt and L. Postle, Progress towards Nash-Williams' conjecture on triangle decompositions, \emph{J. Combin. Theory Ser. B}:  \textbf{146} (2021), 382-416.

\bibitem {joey}J.~Niezen, Sarvate-beam group divisible designs and related multigraph decomposition problems, (Doctoral Dissertation). University of Victoria: Victoria, Canada (2020).

\bibitem{nash}C. St. J. A. Nash-Williams, An unsolved problem concerning decomposition of graphs into triangles. \emph{Combinatorial Theory and its Applications III} (1970),  1179-1183.


\bibitem{woolhouse} W.S.B. Woolhouse. Prize question \#1733. \emph{Lady's and Gentleman's Diary}, (1844). 

\bibitem {Spencer}J. Spencer, Maximal consistent families of triples, \emph{J.
Combin. Theory } \textbf{5 } (1968), 1 - 8.

\bibitem {Bondy}J. A. Bondy and U. S. R. Murty, \emph{Graph Theory with
Applications}, North-Holland, Amsterdam, (1976).

\bibitem{textbook} G. Chartrand, L. Lesniak, and P. Zhang, \emph{Graphs and Digraphs}, Chapman\&Hall/CRC, Boca Raton, (2011).

\bibitem{appel} K. Appel, W. Haken, Every planar map is four colorable, Part I: discharging, \emph{Illinois J. Math.} \textbf{21} (1977), 429-90.

\bibitem{white}
F. Harary, P.C. Kainen, A.J. Schwenk, and A.T. White. A maximal toroidal graph which is not a triangulation. \emph{Math. Scand.}, \textbf{33} (1973), 108-112. 

\bibitem{biedl}
T. Biedl, On triangulating $k$-outerplanar graphs. \emph{Discrete Appl. Math.}, \textbf{181} (2015), 275-179 .


\bibitem{sc3}
L. Markenzon, C.M. Justel, N. Paciornik, Subclasses of $k$-trees: Characterization and recognition, \emph{Discrete Appl. Math.} \textbf{154} (2006), 818-825. 

\bibitem{fort}
M. K. Fort Jr. and G. A. Hedlund, Minimal coverings of pairs by triples, \emph{Pac. J. Appl. Math.} \textbf{8.4} (1958), 709-719.

\bibitem{phan} A. Phan. Petite propagande pour MetaPost: queleques exemples (2). \url{http://www-math.univ-poitiers.fr/~phan/exemples2.html} \emph{Accessed January 2021.}


\end{thebibliography}
\end{document}